%% file: main.tex
\documentclass{article}
\input{preamble.tex}

\begin{document}
\title{When is the Outer Space of a free product CAT(0)?}
\author{Robert Alonzo Lyman}
\maketitle
\begin{abstract}
  Generalizing Culler and Vogtmann's Outer Space for the free group,
  Guirardel and Levitt construct an Outer Space for a free product of groups.
  We completely characterize when this space (or really its simplicial spine)
  supports an equivariant piecewise-Euclidean or piecewise-hyperbolic CAT(0) metric.
  Our results are mostly negative, extending thesis work of Bridson
  and related to thesis work of Cunningham.
  In particular, provided the dimension of the spine is at least three,
  it is never CAT(0).
  Surprisingly, we exhibit one family of free products for which
  the Outer Space is two-dimensional and \emph{does} support an equivariant CAT(0) metric.
\end{abstract}
\input{introduction}
\input{thecomplex}
\input{setupfortheoremA}
\input{specialcase.tex}
\input{maintheorem}
\input{mainnegativetheorem}

\paragraph{Acknowledgments.}
The author is pleased to thank Kim Ruane and Lee Mosher for their interest and mentorship.
Additionally, he is grateful to the anonymous referee for giving this paper a careful reading, as well as for prompting him to flesh out the initial version of the paper.
This material is based upon work supported by the National Science Foundation under Award No. DMS-2202942

\bibliographystyle{alpha}
\bibliography{bib.bib}
\end{document}

%% file: preamble.tex
\usepackage[margin = 1.15 in]{geometry}
\usepackage{amsmath}
\usepackage{amsfonts}
\usepackage{amsthm}
\usepackage{mathtools}
\usepackage[mathscr]{eucal}
\usepackage{hyperref}
\usepackage{tikz-cd}
\usepackage{cleveref}
\usepackage{amssymb}
\usepackage{combelow}

\newtheorem{THM}{Theorem}

\newtheorem{thm}{Theorem}
\numberwithin{thm}{section}
\newtheorem{prop}[thm]{Proposition}
\newtheorem{lem}[thm]{Lemma}
\newtheorem{cor}[thm]{Corollary}
\theoremstyle{definition}

\newtheorem{defn}[thm]{Definition}

\DeclareMathOperator{\out}{Out}
\DeclareMathOperator{\aut}{Aut}

\DeclareMathOperator{\im}{Im}
\DeclareMathOperator{\st}{st}
\DeclareMathOperator{\ad}{ad}
\DeclareMathOperator{\Star}{Star}
\DeclareMathOperator{\Link}{Link}

\DeclareMathOperator{\Stab}{Stab}
\DeclareMathOperator{\Mod}{Mod}
\newcommand{\doublebackslash}{\backslash\mkern-5mu\backslash}

\usepackage{import}

%% file: introduction.tex
\section{Introduction}
This paper grows out of an observation which surprised the author.
We state it as a theorem, and explain why it should be surprising.

\begin{THM}\label{mainpositiveresult} For any nontrivial finite groups $A$ and $B$,
  let $G = A*B*\mathbb{Z}$.
  The group $\out(G)$ of outer automorphisms of $G$ satisfies the following properties.
  \begin{enumerate}
  \item $\out(G)$ acts geometrically on a piecewise-Euclidean CAT$(0)$ polyhedral complex
    of dimension two.
  \item $\out(G)$ contains $\mathbb{Z} \oplus \mathbb{Z}$ and is thus not hyperbolic.
  \item $\out(G)$ has infinitely many ends and is thus relatively hyperbolic.
  \item $\out(G)$ is not a virtual duality group.
  \end{enumerate}
\end{THM}

\begin{proof}
  In actual fact, the positive results of this paper are devoted primarily to the first point
  in \Cref{mainpositiveresult}.
  We prove the remaining statements accepting this one as given for now.
  The fact that the complex so constructed has infinitely many ends,
  and that groups with infinitely many ends are hyperbolic relative to a collection of subgroups
  each with at most one end establishes the third item.
  For a group of virtual cohomological dimension two like $\out(G)$,
  being a virtual duality group is equivalent to one-endedness, establishing the fourth item.
  (When the dimension is more generally $n$, one wants $(n-2)$-connectivity at infinity.)
  Finally,
  If $w$ is an infinite-order element in $A*B$ and $t$ generates a $\mathbb{Z}$ \emph{cofactor}
  of $A*B$ in $A*B*\mathbb{Z}$, the outer classes of the automorphisms $\lambda$ and $\rho$
  of $A*B*\mathbb{Z}$ defined by left- or right-multiplying $t$ by $w$
  and leaving $A*B$ fixed
  are distinct, commuting and have infinite order,
  generating the requisite $\mathbb{Z}\oplus \mathbb{Z}$ subgroup.
\end{proof}

Let $F_n$ denote a free group of rank $n$.
Observe that the group $\out(F_n)$ is finite when $n \le 1$.
The group $\out(F_2)$ is virtually free (i.e.\ has a free subgroup of finite index)
and hence is hyperbolic, has infinitely many ends, acts geometrically on a tree,
does not contain $\mathbb{Z} \oplus \mathbb{Z}$ and is a virtual duality group of dimension one.
However, when $n \ge 3$, we have the following results, in contrast to \Cref{mainpositiveresult}.
\begin{enumerate}
\item The group $\out(F_n)$ is not a CAT$(0)$ group~\cite{Gersten,BridsonVogtmann}.
\item The group $\out(F_n)$ contains $\mathbb{Z} \oplus \mathbb{Z}$ and is in fact
  \emph{thick} in the sense of Behrstock--Dru\c{t}u--Mosher~\cite{BehrstockDrutuMosher}.
  It is therefore neither hyperbolic nor relatively hyperbolic.
\item The group $\out(F_n)$ is one ended~\cite{Vogtmann}.
\item The group $\out(F_n)$ is a virtual duality group of dimension $2n - 3$~\cite{BestvinaFeighn}.
\end{enumerate}

We stumbled upon \Cref{mainpositiveresult}
while attempting to prove that if $G$ is a free product of finite and cyclic groups,
then the only situation in which $\out(G)$ fails to be one-ended
is because it is acts geometrically on a tree (i.e.\ is virtually free).
This project is~\cite{MyOneEnded};
therein we prove that among free products of finite and cyclic groups,
the groups $A * B * \mathbb{Z}$ are the only exception to this rule.

We prove \Cref{mainpositiveresult} by considering a certain complex $L(G)$
on which $\out(G)$ acts properly discontinuously and cocompactly;
it is analogous to the \emph{reduced spine of Outer Space} considered by Culler and Vogtmann
in~\cite{CullerVogtmann}.
We show that $L(G)$ supports an $\out(G)$-equivariant CAT$(0)$ metric
(and a more convenient cellulation)
and exhibit a compact subset of $L$ that separates it into more than one component
with noncompact closure.
(This shows that $L$ and hence $\out(G)$ has more than one end;
to conclude that there are infinitely many and not just two,
observe that two-ended groups do not contain $\mathbb{Z}\oplus\mathbb{Z}$.)

Even compared with results about free products of finite groups,
\Cref{mainpositiveresult} is a surprise.
\begin{enumerate}
\item If $G = A_1 * \cdots * A_n$ is a free product of finite groups with $n \ge 4$,
  Das in his thesis showed that the group $\out(G)$ is thick~\cite{Das}.
\item If $G = C_2 * \cdots * C_2$, where the number of factors is at least $4$,
  Cunningham in his thesis~\cite{Cunningham} showed that another simplicial complex
  defined by McCullough and Miller~\cite{McCulloughMiller} does not support
  an $\out(G)$-equivariant piecewise Euclidean or piecewise-hyperbolic CAT$(0)$ metric.
  This mirrors a result of Bridson's thesis~\cite{Bridson} for $\out(F_n)$ when $n \ge 3$.
\end{enumerate}

In fact, our main result is negative.

\begin{THM}\label{mainnegativeresult}
  Suppose that $T$ is a free splitting of a group $G$,
  whose associated \emph{free factor system} $\mathscr{A}$ decomposes
  $G$ as a free product of groups of the form $G = A_1 * \cdots * A_n * F_k$,
  where the $A_i$ are nontrivial groups and $F_k$ is free of rank $k$.
  The complex $L(T) = L(G,\mathscr{A})$ supports an $\out(G,\mathscr{A})$-equivariant
  piecewise-Euclidean or piecewise-hyperbolic
  CAT$(0)$ metric precisely when
  \begin{enumerate}
  \item The complex $L(T)$ has dimension at most one.
    This happens only for decompositions of the form $G = A$, $A_1 * A_2$,
    $A_1 * A_2 * A_3$, $A_1 * \mathbb{Z}$, $G = \mathbb{Z}$ and $G = F_2$.
  \item The decomposition has the form $G = A_1 * A_2 * \mathbb{Z}$.
  \end{enumerate}
\end{THM}

In other words, for all but finitely many combinatorial possibilities in terms of $n$ and $k$,
the complex $L(T)$ does not support a piecewise-Euclidean or piecewise-hyperbolic
$\out(G,\mathscr{A})$-equivariant CAT$(0)$ metric.
(Actually, the proof also rules out equivariant CAT$(0)$ metrics
which fail to be piecewise-Euclidean or piecewise-hyperbolic
only in that they have a mix of Euclidean and hyperbolic simplices.)

We prove \Cref{mainnegativeresult} by investigating the combinatorial structure of links in $L$.
The definition of the complex $L$ makes sense for an arbitrary \emph{deformation space of trees}
in the sense of Forester and Guirardel--Levitt~\cite{Forester,GuirardelLevitt},
and was first studied in this generality by Clay~\cite{Clay}, who called it $W$.
(Our choice of $L$ follows Culler and Vogtmann.)
Our arguments and perspective on $L$ have the potential to be useful beyond this paper
and are somewhat idiosyncratic:
we work in the quotient graph of groups rather than the Bass--Serre tree.
To do this, we require an innovation of~\cite{MyTrainTracks},
namely the idea of a \emph{map} or a \emph{homotopy equivalence} of graphs of groups,
an idea that builds off of an earlier definition of Bass of a \emph{morphism} of graphs of groups.

Here is the organization of the paper.
\Cref{preliminariessection} contains an exposition of our perspective on the complex $L$,
culminating in the beginning of an understanding of links in $L$
in the particular case when edge groups of some (and in fact any) graph of groups
representing a vertex of $L$ are finite.
We prove \Cref{mainpositiveresult} in \Cref{mainpositivesection},
beginning with the instructive and simplest example
of $C_2 * C_2 * \mathbb{Z}$,
which turns out to contain a certain (non-right-angled) Coxeter group as a subgroup of index four.
The remainder of the paper, \Cref{mainnegativesection},
is given to the proof of \Cref{mainnegativeresult}.

%% file: thecomplex.tex
\section{The complex}\label{preliminariessection}
The purpose of this first section is to introduce the complex $L$
that we will study in what follows.
The definition of $L$ makes sense for an arbitrary cocompact tree action $\rho\colon G \to \aut(T)$.
There is \emph{a priori} a space $L = L(\rho)$ for each action $\rho$,
and we will understand precisely when $\rho$ and $\rho'\colon G \to \aut(T')$
determine the same complex $L$.

For the expert reader, here is a quick definition:
suppose $\rho\colon G \to \aut(T)$ is a cocompact tree action.
We may as well suppose $\rho$ is minimal.
A vertex of $L(\rho)$ is the $G$-equivariant isomorphism class of a minimal, cocompact
action of $G$ on a simplicial tree $T'$
such that there exist $G$-equivariant continuous (not necessarily simplicial) maps $T \to T'$ and $T' \to T$.
It follows that a subgroup of $G$ is elliptic (fixes a point) in $T$ if and only if it fixes a point in $T'$.
We make the further stipulation that every edge of $T$ is \emph{surviving}
in the sense that there is some \emph{reduced} (i.e.\ collapse-minimal) tree in which the edge is not collapsed
(we say the edge ``survives'').
Edges of the complex $L(\rho)$ correspond to collapsing certain orbits of edges in one of the trees
to obtain a tree in the equivalence class of the other.

When $\rho\colon G \to \aut(T)$ is a \emph{free splitting}---that is, edge stabilizers are trivial,
then every vertex of $L(\rho)$ is a free splitting with the same vertex stabilizers.
An edge of $T$ is surviving just when it is either nonseparating in the quotient graph of groups
or when, removing the open edge,
each component of the complement contains a vertex with nontrivial vertex group.
Thus when $G$ is a free group and $\rho\colon G \to \aut(T)$ is a free action,
the complex $L(\rho)$ is the spine of reduced Outer Space for the free group as studied
by Culler and Vogtmann~\cite{CullerVogtmann}.

When $G$ is additionally virtually free and all stabilizers under $\rho$ are finite---that is,
$\rho$ is a splitting of the form $G = A_1 * \cdots A_k * F_n$, where the $A_i$ are finite groups
and $F_n$ is free of some finite rank $n$,
the complex $L$ was studied by Krsti\'c and Vogtmann under another guise~\cite{KrsticVogtmann}
as a fixed-point subcomplex in some Culler--Vogtmann Outer Space.
When the rank of the free group is zero, the space $L$ was studied by McCullough and Miller~\cite{McCulloughMiller}
although is not the space commonly referred to as McCullough--Miller space.

The spine $L(\rho)$ for an arbitrary tree action $\rho\colon G \to \aut(T)$
was first studied by Clay~\cite{Clay}, who called it $W$.
He shows that it is a deformation retract of the deformation space of $\rho$
when the latter is equipped with the weak topology.
Guirardel and Levitt~\cite{GuirardelLevittDeformation}
prove that deformation spaces of trees are always contractible in the weak topology,
so $L$ is a contractible simplicial complex.

Our purposes require a detailed understanding of links of vertices of $L$,
since ultimately we appeal to a success or failure of Gromov's \emph{link condition}
to show whether or not $L$ can be equipped with a CAT$(0)$ metric.
For these purposes, we find it simpler and more illuminating to work in the quotient graph of groups
rather than in the tree $T$.
Part of the length of this section is to put in writing our perspective,
which is the ``graph of groups'' version of Culler and Vogtmann's
original definition of Outer Space as a space of marked (metric) graphs.
Such a definition requires one to make precise what one means by a ``homotopy equivalence''
of a graph of groups, a definition introduced by the author in~\cite{MyTrainTracks}.
We give a full and careful account for completeness,
although the concept of a homotopy equivalence plays a largely ancillary role in this paper.

The main advantage of working in the graph of groups is perhaps subjective:
pictures can be drawn,
and constructions of maps and graphs of groups feel easier;
one is working with concrete and finitary objects rather than infinite trees
equipped with a perhaps-mysterious group action.
Additionally, because the criterion for a tree to be contained in $L$
is a graph-of-groups condition,
it feels easier to make a thorough-going account of this perspective.

Our main result in this section is \Cref{idealedgeproposition},
which describes links of vertices in $L$
for the canonical deformation space of a virtually free group or a free splitting.

\subsection{Graphs of groups and maps thereof}
In this subsection,
we introduce graphs of groups and maps thereof.
We follow notation set in~\cite{MyTrainTracks}.
For additional background on graphs of groups,
the reader is referred to~\cite{Bass, Trees, ScottWall}.

Suppose that a group $G$ acts on a simplicial tree $T$ by simplicial automorphisms.
For each oriented edge $e$ and element $g \in G$,
we have that if $g$ preserves $e$ setwise,
it either preserves it pointwise
or it \emph{inverts} $e$, sending $e$ to $\bar e$,
the edge $e$ with its opposite orientation.
We say that $e$ is \emph{inverted by the $G$-action}
if there exists some $g \in G$ inverting $e$.
The property of being inverted by the $G$-action
is shared by all edges in a single edge orbit.
Therefore by introducing a new vertex at the midpoint of each edge
which is inverted by the $G$-action,
we obtain a new simplicial structure on $T$
which is still respected by the original action of $G$
and which is additionally~\emph{without inversions in edges.}

In this situation---$G$ acting without inversions on a simplicial tree $T$---there is
a construction of the \emph{quotient graph of groups}
which we will make extensive use of.

\paragraph{Graphs of groups.}
But first, what is a graph of groups?
Begin with a graph $|\mathcal{G}|$; that is, a CW-complex of dimension one.
Typically one assumes that $|\mathcal{G}|$ is connected.
For each vertex $v$ and edge $e$,
assign groups $\mathcal{G}_v$ and $\mathcal{G}_e$.
Furthermore, if an oriented edge $e$ of $|\mathcal{G}|$
has initial vertex $v$ and terminal vertex $w$,
we stipulate as part of the data
injective homomorphisms $\iota_e\colon \mathcal{G}_e \to \mathcal{G}_v$
and $\iota_{\bar e}\colon \mathcal{G}_e \to \mathcal{G}_w$.

We return to the situation of a group $G$ acting on a simplicial tree $T$
without inversions in edges.
The quotient $G \backslash T$ naturally inherits the structure of a graph from $T$
because the action is without inversions.

\begin{defn}[Quotient graph of groups]
  The \emph{quotient graph of groups} $G\doublebackslash T$
  depends on several choices.
  The underlying graph of the graph of groups will be $|\mathcal{G}| = G \backslash T$.
  Choose a closed, connected fundamental domain $S$ for the action of $G$ on $T$
  and within it a connected subtree $S_0$ which projects
  homeomorphically under the quotient map $\pi\colon T \to G \backslash T$
  to a spanning tree for the graph $G \backslash T$.
  Thus $S$ contains precisely one preimage $\tilde{e}$ for each edge $e$ of $|\mathcal{G}|$
  and $S_0$ contains precisely one preimage $\tilde{v}$ for each vertex $v$ of $|\mathcal{G}|$.

  The groups $\mathcal{G}_v$ and $\mathcal{G}_e$ for each vertex $v$ and edge $e$
  of $|\mathcal{G}|$ are the stabilizers of $\tilde{v}$ and $\tilde{e}$
  under the action of $G$, respectively.
  When the initial vertex of $\tilde{e}$ is in $S_0$,
  say it is $\tilde{v}$,
  then because the action of $G$ is without inversions in edges,
  we have an inclusion of stabilizers $\Stab(\tilde{e}) \le \Stab(\tilde{v})$;
  this inclusion is the required homomorphism
  $\iota_e\colon \mathcal{G}_e \to \mathcal{G}_v$.

  In the contrary case, when the initial vertex of $\tilde{e}$ is not in $S_0$,
  it is by definition equal to $g^{-1}.\tilde{v}$ for some vertex $\tilde{v} \in S_0$
  and some choice of element $g \in G$.
  If $h \in G$ stabilizes $\tilde{e}$,
  the element $ghg^{-1}$ stablizes $\tilde{v}$,
  so we let the injective homomorphism $\iota_e \colon \mathcal{G}_e \to \mathcal{G}_v$
  be the restriction of the map $h \mapsto ghg^{-1}$ to $\Stab(\tilde{e})$.
\end{defn}

The fundamental theorem of Bass--Serre theory
asserts that \emph{every} connected graph of groups is a quotient graph of groups.
Let us explain.

\paragraph{The fundamental group.}
Given a connected graph of groups $\mathcal{G}$,
we construct a $K(G,1)$ complex $X_{\mathcal{G}}$ for a group $G$
and a map $X_{\mathcal{G}} \to |\mathcal{G}|$ which we will think of as a retraction.
In this way we obtain a group $G$, namely the fundamental group of $X_{\mathcal{G}}$,
and a graph (really a tree) $T$, which is obtained from the universal cover
$\tilde{X}_{\mathcal{G}}$ of $X_{\mathcal{G}}$ by collapsing components of the fibers of the map
$\tilde{X}_{\mathcal{G}} \to X_{\mathcal{G}} \to |\mathcal{G}|$.
The construction of the tree $T$ will play only a background role in this paper.
We refer the reader to~\cite{MyTrainTracks, Bass, Trees, ScottWall} for more information.

Explicitly, for each vertex $v$ and edge $e$ of $|\mathcal{G}|$,
choose $K(\mathcal{G}_v,1)$ and $K(\mathcal{G}_e,1)$ complexes
$X_v$ and $X_e$, each with a single vertex.
Because they are $K(\pi,1)$ spaces,
the homomorphisms $\iota_e\colon \mathcal{G}_e \to \mathcal{G}_v$
and $\iota_{\bar e} \colon \mathcal{G}_e \to \mathcal{G}_w$
are realized by continuous, cellular maps
$i_e \colon X_e \to X_v$ and $i_{\bar e}\colon X_e \to X_w$.

The space $X_{\mathcal{G}}$ is the quotient of the disjoint union
\[
  \coprod_v X_v \sqcup \coprod_e \left(X_e \times [0,1]\right)
\]
by attaching the subspaces $X_e \times \{0\}$ and $X_e \times \{1\}$
in the pattern of $|\mathcal{G}|$.
That is, each point $(x,0)$ is identified with $i_e(x)$,
and $(x,1)$ with $i_{\bar e}(x)$.
This is a CW-complex.
It turns out to be a $K(G,1)$ space.

Anyway, since each $X_v$ and $X_e$ have one vertex $\star_v$ and $\star_e$
and the maps $i_e$ and $i_{\bar e}$ as $e$ varies are cellular,
the subspace of $X_{\mathcal{G}}$ comprising those points $\star_v$ and $(\star_e, t)$
for $t \in [0,1]$ is connected and homeomorphic to $|\mathcal{G}|$.
Collapsing each $X_v$ and $X_e \times \{t\}$ to a point
yields a map $X_{\mathcal{G}} \to |\mathcal{G}|$;
we think of it as a retraction.

\begin{defn}
  The \emph{fundamental group of the graph of groups $\mathcal{G}$}
  is the fundamental group of $X_{\mathcal{G}}$.

  Choose a basepoint $\star$ in the image of $|\mathcal{G}|$.
  Indeed, for simplicity, suppose that $\star$ is a vertex.
  By cellular approximation, each loop based at $\star$,
  and more generally each path with endpoints on $|\mathcal{G}|$,
  may be homotoped into the $1$-skeleton of $X_{\mathcal{G}}$,
  and thus (after a further homotopy rel endpoints)
  corresponds to a \emph{graph of groups edge path}
  of the form
  \[
    g_0e_1g_1\ldots e_kg_k,
  \]
  where each $g_i$ is an element of some vertex group $\mathcal{G}_v$,
  (that is, a loop in some vertex \emph{space} $X_v$),
  each $e_i$ is an edge of $|\mathcal{G}|$,
  and the concatenation $e_1 \ldots e_k$ is an edge path
  (that is, the oriented image of some path $[0,1] \to |\mathcal{G}|$) in $|\mathcal{G}|$.

  Homotopy rel endpoints of graph of groups edge paths in $X_{\mathcal{G}}$
  has a combinatorial shadow in $\mathcal{G}$
  without reference to $X_{\mathcal{G}}$:
  it is generated by the following operations.
  \begin{enumerate}
  \item Introduce or remove a subpath of the form $e 1 \bar e$,
    where $1$ denotes the identity element of the relevant vertex group.
  \item Replace a subpath of the form $g\iota_e(h)e g'$
    with one of the form $ge\iota_{\bar e}(h)g'$ or vice versa,
    where $h \in \mathcal{G}_e$ is an edge group element.
  \end{enumerate}
\end{defn}

\paragraph{Maps of graphs of groups.}
The motivation for the definition of a \emph{map} of a graph of groups is twofold:
maps of graphs of groups should reflect (twisted) equivariant maps of Bass--Serre trees,
and they should reflect at least some homotopy classes of maps of the $K(G,1)$
complexes described above,
namely those maps $f_X \colon X_{\mathcal{G}} \to X_{\mathcal{H}}$
which are \emph{fiber-preserving,}
in the sense that they induce a map $|f|\colon |\mathcal{G}| \to |\mathcal{H}|$
making the following diagram commute
\[
  \begin{tikzcd}
    X_{\mathcal{G}} \ar[r, "f_X"] \ar[d] & X_{\mathcal{H}} \ar[d] \\
    {|\mathcal{G}|} \ar[r, "|f|"] & {|\mathcal{H}|}.
  \end{tikzcd}
\]
It is this latter perspective we take up now;
for the trees perspective the reader is referred to~\cite{MyTrainTracks, Bass}.

By cellular approximation,
the map $f_X$ is homotopic to a cellular map,
and after a further homotopy, we may assume that
for each edge $e$ of $|\mathcal{G}|$,
the image $f_X(e)$ corresponds to a graph of groups edge path.
The map $f_X$ is thus homotopic to a map with the following properties reflected
in $\mathcal{G}$ and $\mathcal{H}$.

\begin{defn}[Map of graphs of groups]
  A \emph{map} of graphs of groups $f\colon \mathcal{G} \to \mathcal{H}$
  is the data of the following.
  \begin{enumerate}
  \item A continuous map $|f|\colon |\mathcal{G}| \to |\mathcal{H}|$
    sending vertices to vertices
    and edges to edge paths.
  \item For each vertex $v$ of $|\mathcal{G}|$,
    a homomorphism $f_v\colon \mathcal{G}_v \to \mathcal{H}_{|f|(v)}$.
  \item For each edge $e$, a graph of groups edge path
    \[
      f(e) = g_0 e_1 g_1 \ldots e_n g_n
    \]
    such that the underlying edge path $e_1 \ldots e_n$
    is the image $|f|(e)$.
    We allow graph of groups edge paths to be reduced to a single vertex group element.
  \item For each edge $e$ and each edge $e_i$ appearing in the edge path $|f|(e)$,
    a homomorphism $f_{e,e_i}\colon \mathcal{G}_e \to \mathcal{H}_{e_i}$
    such that the following diagram commutes
    \[
      \begin{tikzcd}[column sep = huge]
        \mathcal{G}_v \ar[d, "f_v"]
        & & \mathcal{G}_e \ar[ll, "\iota_e"'] \ar[rr, "\iota_{\bar e}"]
        \ar[ld, "f_{e,e_i}"'] \ar[rd, "f_{e,e_{i+1}}"]
        & & \mathcal{G}_w \ar[d, "\ad(g_n) f_w"] \\
        \mathcal{H}_{f(v)} \ldots & \mathcal{H}_{e_i} \ar[l, "\ad(g_{i-1})\iota_{e_i}"']
        \ar[r, "\iota_{\bar e_i}"] & \mathcal{H}_{v_i} & \mathcal{H}_{e_{i+1}}
        \ar[l, "\ad(g_i)\iota_{e_{i+1}}"'] \ar[r, "\iota_{\bar e_{i+1}}"] &
        \ldots \mathcal{H}_{f(w)}.
      \end{tikzcd}
    \]
    where $\ad(g)$ denotes the inner automorphism $x \mapsto gxg^{-1}$.
    This diagram has an appropriate modification when $|f|(e)$ contains no edges.
  \end{enumerate}
  The relevance of these diagrams is that $f$ acts appropriately
  with respect to homotopies of paths: if $h \in \mathcal{G}_e$ is an edge group element,
  we compute
  \begin{align*}
    f_v\iota_e(h)f(e) &= f_v\iota_e(h)g_0e_1g_1\ldots e_ng_n
    = g_0\iota_{e_1}f_{e,e_1}(h)e_1g_1\ldots e_ng_n \\ &=
    g_0e_1\iota_{\bar e_1}f_{e,e_1}(h)g_1 \ldots e_n g_n = \cdots = f(e)f_w\iota_{\bar e}(h).
  \end{align*}
\end{defn}

If $T$ is a $G$-tree, $S$ is an $H$-tree, and $\Phi\colon G \to H$ is a homomorphism,
a continuous map $\tilde{f}\colon T \to S$ is \emph{$\Phi$-twisted equivariant}
if for each $g \in G$ and each $x \in T$,
we have that
\[
  \tilde{f}(g.x) = \Phi(g).\tilde{f}(x).
\]
If $\tilde{f}$ is a $\Phi$-twisted equivariant map
sending vertices to vertices and edges to edge paths,
then there is an induced map of graphs of groups
$f\colon G\doublebackslash T \to H\doublebackslash S$.
Conversely, each map $f\colon \mathcal{G} \to \mathcal{H}$
induces a map on fundamental groups and a twisted equivariant map of Bass--Serre trees.
See~\cite{MyTrainTracks} for details.

If a map of graphs of groups is a pair of maps
$f_X$ and $|f|$,
a \emph{homotopy} of maps of graphs of groups is simply a homotopy of maps
$f_X \colon X_{\mathcal{G}} \to X_{\mathcal{H}}$
and $|f|\colon |\mathcal{G}| \to |\mathcal{H}|$
such that each resulting square commutes.
Although intermediate pairs of maps do not induce
maps with all of the structure above
(for example, the image of a vertex need not be a vertex,
nor the image of an edge an edge path),
homotopy of maps has a combinatorial reflection in $\mathcal{G}$ and $\mathcal{H}$ as well.

\begin{defn}[Homotopy of maps of graphs of groups]
  A pair of maps $\mathcal{G} \to \mathcal{H}$ are \emph{homotopic}
  if one can be transformed into the other by a finite sequence of the following moves.
  \begin{enumerate}
  \item Replace the edge path $f(e)$ with a path homotopic to it rel endpoints
    and replace the homomorphisms $f_{e,e_i}$ with new homomorphisms
    satisfying the compatibility condition above.
  \item If $g$ is an element of $\mathcal{H}_{f(v)}$,
    replace the map $f_v$ with $\ad(g)f_v$.
    Furthermore for each oriented edge $e$ with initial vertex $v$,
    replace the edge path $f(e)$ with $gf(e)$.
  \item Let $e$ be an oriented edge of $\mathcal{H}$
    with initial vertex $w$ and terminal vertex $f(v)$
    such that $f_v(\mathcal{G}_v) \le \iota_{\bar e}(\mathcal{H}_e)$.
    Change the map $f$ by a homotopy with support in a neighborhood of $v$
    by pulling the image of $v$ across $e$.
    The new map, which we shall call $f'$, satisfies
    $f'(v) = w$
    and will satisfy that $f'(e') = ef(e')$ for each oriented edge $e'$ in $\mathcal{G}$
    with initial vertex $v$.
    (Clearly this statement needs the appropriate amendation when $e'$ forms a loop.)
    The new homomorphism $f'_v$ is equal to $\iota_{e}\iota_{\bar e}^{-1}f_v$.
    That is to say,
    it is accomplished by viewing $f_v(\mathcal{G}_v)$ as a subgroup of $\mathcal{H}_e$
    and then mapping it to $\mathcal{H}_w$.
    For each oriented edge $e'$ as above, we add a new homomorphism
    $f'_{e',e}\colon \mathcal{G}_{e'} \to \mathcal{H}_e$
    defined as $f'_{e',e} = \iota_{\bar e}^{-1}f_v\iota_{e'}$.
    It would be a good exercise to verify the compatibility conditions.
  \end{enumerate}
\end{defn}

A map $f\colon \mathcal{G} \to \mathcal{H}$ is a \emph{homotopy equivalence}
if, as usual, there exists a map $g\colon \mathcal{H} \to \mathcal{G}$
such that each of the double compositions $fg$ and $gf$ are homotopic to the identity.
If $f\colon \mathcal{G} \to \mathcal{H}$ is a homotopy equivalence,
then the induced map on fundamental groups is an isomorphism
and the corresponding Bass--Serre trees belong to the same \emph{deformation space}
in the sense of Forester and Guirardel--Levitt---but more about deformation spaces later.

A map $f\colon \mathcal{G} \to \mathcal{H}$ is an \emph{isomorphism}
if the map $|f|$ of underlying graphs is an isomorphism,
and each homomorphism $f_v$ and $f_{e,e'}$ (note that $f(e)$ contains a single edge)
is an isomorphism.
Beyond satisfying the compatibility conditions,
we make no stipulation about the vertex group elements on each path $f(e)$.

\subsection{Deformation spaces and their spines}

The purpose of this subsection is to introduce our perspective on deformation spaces of trees.
It is the graph of groups analogue of Culler--Vogtmann's original definition of Outer Space
for the free group~\cite{CullerVogtmann}.

So, fix once and for all a group $G$ and a graph of groups $\mathbb{G}$
with fundamental group $G$.
We assume that $|\mathbb{G}|$ is a finite, connected graph.

A \emph{metric} on a graph of groups $\mathcal{G}$ is simply an assignment
of positive \emph{lengths} $\ell(e)$ to each edge $e$ of $\mathcal{G}$.

\begin{defn}
  A \emph{marked metric graph of groups} is a pair
  $(\mathcal{G},\sigma)$, where $\mathcal{G}$
  is a finite, connected graph of groups equipped with a metric,
  and $\sigma\colon \mathbb{G} \to \mathcal{G}$ is a homotopy equivalence
  called the \emph{marking.}
  A vertex $v$ of $\mathcal{G}$ is \emph{inessential}
  if for each edge $e$ incident to $v$,
  the edge-to-vertex group inclusion $\mathcal{G}_e \to \mathcal{G}_v$
  is surjective (i.e.\ an isomorphism).
  We assume that $\mathcal{G}$ has no inesential valence-one or valence-two vertices.

  Two marked metric graphs of groups $(\mathcal{G},\sigma)$ and $(\mathcal{G}',\sigma')$
  are equivalent if there is an isomorphism $h\colon \mathcal{G} \to \mathcal{G}'$
  preserving edge lengths
  and satisfying that the following diagram commutes up to homotopy
  \[
    \begin{tikzcd}[row sep = tiny]
      & \mathcal{G} \ar[dd, "h"] \\
      \mathbb{G} \ar[ur, "\sigma"] \ar[dr, "\sigma'"] & \\
      & \mathcal{G}'.
    \end{tikzcd}
  \]
\end{defn}

The \emph{deformation space} of $\mathbb{G}$, call it $\mathscr{T}(\mathbb{G})$, is, as a set,
the collection of equivalence classes of $\mathbb{G}$-marked metric graphs of groups.
In general there are three topologies on $\mathscr{T}(\mathbb{G})$,
the \emph{axes topology} for which the hyperbolic translation length of hyperbolic elements
(or conjugacy classes) are continuous functions on $\mathscr{T}(\mathbb{G})$,
the \emph{equivariant Gromov--Hausdorff topology}
which in a very rough sense describes that the ``shape'' of the trees doesn't change
too much as one moves to a nearby tree
(these agree for $\mathscr{T}(\mathbb{G})$ when the Bass--Serre tree of $\mathbb{G}$
is ``irreducible''---i.e.\ contains a free group acting properly discontinuously)
and the \emph{weak topology,} which we now describe.

We will work only with the \emph{weak topology:}
notice that we may freely vary the assignment of edge lengths to $\mathcal{G}$.
Doing so describes, it turns out,
an embedding (of sets at this point)
of the positive cone in $\mathbb{R}^M$ into $\mathscr{T}(\mathbb{G})$,
where $M$ is the number of edges of $|\mathcal{G}|$.
The whole of $\mathscr{T}(\mathbb{G})$ is a disjoint union of these cones,
and we describe the weak topology by saying that a set $U \subset \mathscr{T}(\mathbb{G})$
is closed if and only if its intersection with each \emph{closed cone}
(to be defined presently) is closed.

In general, these three topologies are really different.
In the particular case where $G$ is a virtually free group
and $\mathbb{G}$ is a finite graph of finite groups,
they all agree.

There is a natural action of $\mathbb{R}_+$ on $\mathscr{T}(\mathbb{G})$
by scaling the length of each edge of a marked metric graph of groups
by the same positive real number.
The quotient by this action, the \emph{projectivized deformation space}
$\mathscr{PT}(\mathcal{G})$,
also admits three topologies.
One might be tempted to identify $\mathscr{PT}(\mathbb{G})$
by for example normalizing the sum of the lengths of edges according to some scheme.
This is an innocent thing to do for a virtually free group $G$
where $\mathbb{G}$ is a finite graph of finite groups,
but caution is needed in general:
this identification of sets need not yield an embedding of topological spaces.

\begin{defn}[Collapsible edges, collapses and their homotopy inverses, closed cones]
  An edge $e$ of $\mathcal{G}$ is \emph{collapsible}
  if its endpoints are distinct, say $v$ and $w$,
  and at least one of the edge-to-vertex group inclusions
  $\mathcal{G}_e \to \mathcal{G}_v$ or $\mathcal{G}_e \to \mathcal{G}_w$
  is surjective (i.e.\ an isomorphism).

  In this situation, define a new graph of groups $\mathcal{G}'$ as follows.
  The graph $|\mathcal{G}'|$ is obtained from $|\mathcal{G}|$ by collapsing the edge $e$
  to a point.
  The vertex group of the common image of $v$ and $w$
  is $\mathcal{G}_e \cong \mathcal{G}_v \cong \mathcal{G}_w$
  if \emph{both} edge-to-vertex group inclusions are isomorphisms in $\mathcal{G}$.
  If one of them, say $\mathcal{G}_e \to \mathcal{G}_v$ is not an isomorphism,
  then the new vertex group is $\mathcal{G}_v$.
  All other edge and vertex groups are equal to what they were in $\mathcal{G}$.
  The edge-to-vertex group inclusions are either what they were in $\mathcal{G}$,
  or of the form $\mathcal{G}_{e'} \to \mathcal{G}_w \cong \mathcal{G}_e \to \mathcal{G}_v$.

  There is a natural map of graphs of groups $k\colon \mathcal{G} \to \mathcal{G}'$.
  We call the map $k$ a \emph{collapse.}
  It would make a good exercise to write it down carefully.
  The map $k$ is a homotopy equivalence, basically because the edge $e$ was collapsible.
  One loose description of a homotopy inverse $f$ is as follows:
  it is ``the identity'' on the ``common'' edges and vertices
  of $\mathcal{G}$ and $\mathcal{G}'$
  except for those edges $e'$ which are incident to $w$ in $\mathcal{G}$,
  on which the map $f\colon \mathcal{G}' \to \mathcal{G}$
  is $e$ (with the proper orientation) followed by $e'$.
  The compatibility conditions essentially follow from the fact that $e$ was collapsible.

  Anyway, because $k$ is a homotopy equivalence,
  if $\sigma\colon \mathbb{G} \to \mathcal{G}$ is a marking,
  we may give $\mathcal{G}'$ the marking $k\sigma$.
  In $\mathscr{T}(\mathbb{G})$, we may think of points of the (open) cone
  determined by $(\mathcal{G}', k\sigma)$ as being a codimension-one face
  of the cone determined by $(\mathcal{G},\sigma)$,
  namely the face where the collapsed edge has been assigned a length of zero.

  Beginning with a marked graph of groups $(\mathcal{G},\sigma)$,
  we define the \emph{closed cone} in $\mathscr{T}(\mathbb{G})$ determined by $(\mathcal{G},\sigma)$
  recursively by beginning with the open cone of $(\mathcal{G},\sigma)$,
  including the face determined by any collapsible edge of $\mathcal{G}$,
  then recursing by adding the faces of the faces and so on.
  
  Thus in $\mathscr{T}(\mathbb{G})$ equipped with the weak topology,
  for any choice of a metric on $\mathcal{G}'$
  and on $\mathcal{G}$, neighborhoods of $(\mathcal{G}', k\sigma)$
  contain points of the cone determined by $(\mathcal{G},\sigma)$.
  If one considers instead projective classes of metrics,
  the same statement is true of simplices in $\mathscr{PT}(\mathbb{G})$.
\end{defn}

The space $\mathscr{PT}(\mathbb{G})$ is not, strictly speaking, a simplicial complex,
but rather (in the weak topology) like a simplicial complex with certain faces removed.
It has a simplicial \emph{spine} to which $\mathscr{PT}(\mathbb{G})$ deformation retracts
(in the weak topology) and which we now describe.

\begin{defn}[The spine of $\mathscr{PT}(\mathbb{G})$]
  The \emph{spine} $K(\mathbb{G})$ of $\mathscr{PT}(\mathbb{G})$ is a simplicial complex.
  The vertices of $K(\mathbb{G})$ are equivalence classes of
  marked graphs of groups $(\mathcal{G},\sigma)$
  (there is no metric), where two marked graphs of groups are equivalent,
  again, if there is an isomorphism $h\colon \mathcal{G}\to \mathcal{G}'$
  preserving the homotopy classes of the markings as before.
  Because there is no metric, there is no stipulation that $h$ preserve edge lengths.
  We remind the reader that $\mathcal{G}$ is not allowed to have
  inessential valence-one or valence-two vertices,
  and that the graph $|\mathcal{G}|$ should be finite and connected.

  A collection of marked graphs of groups $\tau_0,\ldots,\tau_n$ spans an ordered $n$-simplex
  when there are collapses $\tau_n \to \tau_{n-1} \to \cdots \to  \tau_0$
  (or, strictly speaking, between representatives of their equivalence classes).

  In the weak topology, we may embed the vertices of $K(\mathbb{G})$ in
  $\mathscr{PT}(\mathbb{G})$ by sending $(\mathcal{G},\sigma)$
  to the same marked graph of groups equipped with the projective class of a metric
  in which (for instance) every edge has the same length.
  That is, if we think of a cone in $\mathscr{T}(\mathbb{G})$
  as yielding an open simplex in $\mathscr{PT}(\mathbb{G})$,
  pick out its barycenter.
  The closure of this open simplex in $\mathscr{PT}(\mathbb{G})$
  is homeomorphic to a closed simplex in $\mathbb{R}^M$
  minus certain faces.
  If a face belongs to this ``closed'' simplex,
  then its barycenter is a vertex of the spine,
  and this face may be thought of as being obtained by letting certain barycentric coordinates
  go to zero.

  Take the first barycentric subdivision of the ``abstract'' closed simplex in $\mathbb{R}^M$.
  A simplex of this barycentric subdivision corresponds to a simplex of the spine
  $K(\mathbb{G})$ just when
  all of the vertices of the ``abstract'' closed $(M-1)$ simplex
  belong to the ``actual'' closed simplex in $\mathscr{PT}(\mathbb{G})$.

  So for example, if the simplex corresponding to $(\mathcal{G},\sigma)$
  is a triangle with the edges included but not the vertices,
  then the corresponding portion of the spine is a tripod,
  with the center of the tripod corresponding to the barycenter of $(\mathcal{G},\sigma)$,
  and the three leaves corresponding to the other barycenters.
  If we imagine adding one of the missing vertices into $\mathscr{PT}(\mathbb{G})$,
  the portion of the spine adds that vertex of the triangle
  containing the barycenters of the edges incident to that new vertex,
  the vertex itself, and the barycenter of the original triangle.
\end{defn}

It should be essentially clear that
$\mathscr{PT}(\mathbb{G})$ deformation retracts onto $K(\mathbb{G})$
when the former is equipped with the weak topology;
see~\cite{CullerVogtmann,McCulloughMiller} for details.
There is a further deformation retraction---$L(\mathbb{G})$---studied by Clay~\cite{Clay}
which we now describe.

\begin{defn}[Surviving edges, reduced spine]
  The spine $K(\mathbb{G})$ is the geometric realization of a poset
  whose elements are marked graphs of groups
  and whose partial order is the relations $\tau_0 \prec \tau_1$
  where $\tau_1$ collapses onto $\tau_0$.
  This partial order has minimal elements;
  these are called \emph{reduced} marked graphs of groups.
  (Indeed, the number of edges goes down under collapse.)

  If $(\mathcal{G},\sigma)$ is a marked graph of groups representing a vertex
  of $K(\mathbb{G})$,
  say that an edge $e$ of $|\mathcal{G}|$ is \emph{surviving}
  if there exists a reduced collapse $\mathcal{G}'$
  of $\mathcal{G}$ in which the edge $e$ survives---i.e.\ is not collapsed.

  The \emph{reduced spine $L(\mathbb{G})$ of the deformation space}
  is the subcomplex of $K(\mathbb{G})$ spanned by those marked graphs of groups
  all of whose edges are surviving.
\end{defn}

By collapsing edges which are not surviving,
one has a deformation retraction from $K(\mathbb{G})$ to $L(\mathbb{G})$
(This result is a theorem of Clay~\cite{Clay}).
When $\mathbb{G}$ is a graph (with trivial vertex and edge groups),
so that $G$ is a free group of some finite rank $n$,
then every reduced marked graph (of groups)
is a \emph{rose} with $n$ petals;
in particular, none of its edges are separating.
Since a separating edge remains separating in any collapse in which it survives,
we see that $L(\mathbb{G})$ really is the complex $L$ considered by Culler--Vogtmann
in~\cite{CullerVogtmann}.
That is, vertices of $L$ are marked rank-$n$ graphs with no valence-one or valence-two
vertices and no separating edges.

Clay~\cite{Clay} and Guirardel--Levitt~\cite{GuirardelLevitt} give a characterization
of when a graph of groups has all edges surviving.

\begin{defn}
  A \emph{shelter} in a graph of groups $\mathcal{G}$ is a subgraph of groups $\mathcal{S}$
  of one of the following three forms.
  \begin{enumerate}
  \item (Segment shelter.) The graph $|\mathcal{S}|$ is homemorphic to an interval.
    Every valence-two vertex of $|\mathcal{S}|$ is inesential in $\mathcal{S}$,
    but neither valence-one vertex is inessential.
  \item (Basepointed loop shelter.) The graph $|\mathcal{S}|$ is homeomorphic to a circle.
    There is a distinguished vertex $v$, neither of whose edge-to-vertex group inclusions
    is surjective, while every other vertex is inessential.
  \item (Oriented loop shelter.) The graph $|\mathcal{S}|$ is homeomorphic to a circle,
    and there is an orientation on the circle such that if $e$ is positively oriented,
    then the initial edge-to-vertex group inclusion is surjective,
    while the other may or may not be.
  \end{enumerate}
\end{defn}

We will say that an oriented loop shelter is \emph{flat} if may be oriented in either
direction: that is, \emph{every} edge-to-vertex group inclusion on the loop is surjective.
If an oriented loop shelter is not flat, we call it \emph{ascending.}
A deformation space is \emph{ascending} if it contains a marked graph of groups
$(\mathcal{G},\sigma)$ with an ascending oriented loop shelter,
and \emph{non-ascending} otherwise:
that is, in a non-ascending deformation space, all oriented loop shelters are flat.

Clay proves~\cite[Proposition 1.13]{Clay} that an edge of $\mathcal{G}$ is surviving
if and only if it is contained in a shelter.

\begin{defn}[Deformation group]
  If $G = \pi_1(\mathbb{G})$, notice that each homotopy equivalence
  $\mathbb{G} \to \mathbb{G}$ induces, after choosing basepoints,
  an automorphism of $G$.
  If we consider the (free) homotopy classes of these homotopy equivalences,
  we obtain a group, which we denote $\Mod(\mathbb{G})$,
  the \emph{deformation group} or \emph{modular group}
  of the deformation space of $\mathbb{G}$.
  It is naturally a subgroup of $\out(G)$.
  If $\tau = (\mathcal{G},\sigma)$ is a $\mathbb{G}$-marked graph of groups,
  the marking $\sigma$ provides a canonical identification
  of $\Mod(\mathcal{G})$ with $\Mod(\mathbb{G})$,
  so this group is really a property of the deformation space.
  In fact, $\Mod(\mathbb{G})$ naturally acts on $\mathscr{T}(\mathbb{G})$
  (and $\mathscr{PT}(\mathbb{G})$, $K(\mathbb{G})$ and $L(\mathbb{G})$)
  on the right by the rule that the homotopy class of $f\colon \mathbb{G} \to \mathbb{G}$
  sends $\tau = (\mathcal{G},\sigma)$ to
  \[ \tau.f = (\mathcal{G}, \sigma f). \]
  Because the combinatorial type of $\mathcal{G}$ is preserved,
  when considered on $K(\mathbb{G})$ or $L(\mathbb{G})$,
  this action is by simplicial automorphisms.

  We say that the deformation space of $\mathbb{G}$ is \emph{canonical}
  when $\Mod(\mathbb{G}) = \out(\pi_1(\mathbb{G}))$.
  This happens when elliptic subgroups of $\mathbb{G}$ are ``algebraically determined''
  by $G = \pi_1(\mathbb{G})$:
  for example, if elliptic subgroups (hence vertex and edge groups) of $\mathbb{G}$
  are finite, then $G$ is virtually free
  and the deformation space of $\mathbb{G}$ is canonical,
  because finite groups always act elliptically on trees,
  and any automorphism sends finite subgroups to finite subgroups.

  Another canonical deformation space is associated to the \emph{Grushko decomposition}
  of a finitely generated group $G = A_1 * \cdots * A_k * F_n$,
  where the $A_i$ are freely indecomposable (and finitely generated)
  and $F_n$ is free of finite rank $n$.
\end{defn}

The interest in deformation spaces of trees
is that they are always contractible
(in the weak topology)~\cite{CullerVogtmann,GuirardelLevitt,Clay},
proving that $\Mod(\mathbb{G})$ is of type $F_\infty$
in situations when it acts properly discontinuously and cocompactly
on $K(\mathbb{G})$ or $L(\mathbb{G})$.
In many such cases, the group $\Mod(\mathbb{G})$ is furthermore
virtually torsion free,
and thus virtually of type $F$.

One such case where this happens is when $\mathscr{T}(\mathbb{G})$
is the canonical deformation space of a virtually free group.
When $G = \pi_1(\mathbb{G})$ is virtually free and a free product
(i.e. a ``plain'' group; a free product of finitely many finite groups
and a free group of finite rank),
Krsti\'c and Vogtmann show that the dimension of $L(\mathbb{G})$
is equal to the virtual cohomological dimension of $\out(G)$.
This is false for general virtually free groups.

\subsection{Links in \texorpdfstring{$K$}{K} and \texorpdfstring{$L$}{L}}
The purpose of this subsection is to begin the combinatorial study
of links in $L(\mathbb{G})$ from a graph-of-groups point of view.
In the case of the canonical deformation space of a free group
(that is, Culler--Vogtmann Outer Space),
this study was undertaken by Culler and Vogtmann~\cite{CullerVogtmann}.
A slightly different perspective on the complex $L(\mathbb{G})$
when $G = \pi_1(\mathbb{G})$ is virtually free appears in~\cite{KrsticVogtmann}.

\begin{defn}[Set of directions]
  Let $v$ be a vertex of a marked graph of groups $(\mathcal{G},\sigma)$.
  We write $\st(v)$ for the set of oriented edges with initial vertex $v$.
  The set of \emph{directions} at $v$ is
  \[
    D_v = \coprod_{e \in \st(v)} \mathcal{G}_v / \iota_e(\mathcal{G}_e) \times \{e\}.
  \]
  For the reader more familiar with Bass--Serre theory,
  note that this set is in equivariant bijection with the set of oriented edges
  incident to some and hence any lift of $v$ to the Bass--Serre tree of $\mathcal{G}$.
  There is an obvious action of $\mathcal{G}_v$ on $D_v$;
  each orbit is an oriented edge $e \in \st(v)$.

  If $d \in D_v$ is a direction, we will write $\mathcal{G}_d$ for the stabilizer
  of $d$ under the action of $\mathcal{G}_v$.
  When $d$ is in the orbit $e \in \st(v)$, we have that $\mathcal{G}_d$
  is a conjugate of $\mathcal{G}_e$.
\end{defn}

\begin{defn}[Ideal edge]
  Let $\alpha \subset D_v$ be a subset of $D_v$
  and $\mathcal{G}_\alpha$ a subgroup of $\mathcal{G}_v$.
  We say that the pair $(\alpha, \mathcal{G}_\alpha)$
  is an \emph{ideal edge $\alpha$ with stabilizer $\mathcal{G}_\alpha$}
  when for each direction $d \in \alpha$,
  we have that $\mathcal{G}_v.d \cap \alpha = \mathcal{G}_\alpha.d$.
  That is, a direction is contained in an ideal edge $\alpha$ if and only if
  its orbit under $\mathcal{G}_\alpha$ is contained in $\alpha$.
\end{defn}

The group $\mathcal{G}_v$ acts on the set of ideal edges $(\alpha, \mathcal{G}_\alpha)$
based at $v$ by the rule
$g.(\alpha, \mathcal{G}_\alpha) = (g.\alpha, g\mathcal{G}_\alpha g^{-1})$.
We say that ideal edges $\alpha$ and $\alpha'$ are \emph{equivalent}
if they belong to the same $\mathcal{G}_v$-orbit.

The collection of equivalence classes of ideal edges is partially ordered,
where we say that $[\alpha, \mathcal{G}_\alpha] \sqsubset [\beta, \mathcal{G}_\beta]$
or say that $[\alpha,\mathcal{G}_\alpha]$ is \emph{nested} in $[\beta,\mathcal{G}_\beta]$
if there exist representatives satisfying $\alpha \subset \beta$
and $\mathcal{G}_\alpha \le \mathcal{G}_\beta$.
Furthermore, we say that $[\alpha, \mathcal{G}_\alpha]$ and $[\beta, \mathcal{G}_\beta]$
are \emph{disjoint} if they are not related by $\sqsubset$
and additionally $\mathcal{G}_v.\alpha \cap \mathcal{G}_v.\beta = \varnothing$.

In the special case where $\mathcal{G}_\alpha = \mathcal{G}_v$,
observe that $(D_v - \alpha, \mathcal{G}_v)$ is also an ideal edge.

Two (equivalence classes of) ideal edges based at a common vertex $v$ are \emph{compatible}
when they are disjoint or nested.
We also say that ideal edges based at different vertices are compatible.

An \emph{oriented ideal forest} $\Phi^+$ in $\mathcal{G}$ is a collection of equivalence classes
of ideal edges which are pairwise compatible
and satisfies the condition
that if $[\alpha, \mathcal{G}_\alpha = \mathcal{G}_v]$ is an ideal edge in $\Phi^+$,
then $\Phi^+$ does not contain $[D_v - \alpha, \mathcal{G}_v]$.

If $[\alpha, \mathcal{G}_v]$ is an ideal edge in an oriented ideal forest $\Phi^+$,
observe that every ideal edge which is compatible with $[\alpha, \mathcal{G}_v]$
is also compatible with $[D_v - \alpha, \mathcal{G}_v]$.
We call the operation of replacing $\Phi^+$ with
$\Phi^+ \cup \{[D_v - \alpha, \mathcal{G}_v]\} - \{[\alpha, \mathcal{G}_v]\}$
\emph{flipping the orientation of the ideal edge $\alpha$.}
Flipping orientations of ideal edges generates an equivalence relation
on oriented ideal forests;
an equivalence class is an \emph{ideal forest.}

An \emph{oriented collapsible forest} $\Psi^+$ in $\mathcal{G}$ is a collection of edges of $|\mathcal{G}|$,
each of which is collapsible and equipped with an orientation such that the following conditions hold.
\begin{enumerate}
\item Each component of $\Psi^+$ is a tree (i.e.\ a graph of groups with no loops).
\item The terminal edge-to-vertex group inclusion for each oriented edge in $\Psi^+$ is surjective.
\item Each component $C$ of $\Psi^+$ is oriented away from a distinguished vertex $v_C$.
\end{enumerate}
Call a vertex $v$ of $C$ \emph{heavy} if no edge-to-vertex group inclusion between $v$ and an edge of $C$
is surjective.
If a component $C$ of an oriented collapsible forest $\Psi^+$ has a heavy vertex,
it must have only one, and that vertex must be $v_C$.
If no vertex of $C$ is heavy, it may be the case that reversing the orientation
of some edge $e$ of $\Psi^+$ yields another oriented collapsible forest.
This happens if and only if both edge-to-vertex group inclusions involving $e$ are surjective.
The operation of flipping the orientation on $e$ generates an equivalence relation
on oriented collapsible forests;
an equivalence class is a \emph{collapsible forest.}

\paragraph{Blowing up ideal edges.}
Suppose $\alpha \subset D_v$ is an ideal edge based at a vertex
of the graph of groups $\mathcal{G}$.
We describe a new graph of groups $\mathcal{G}^\alpha$.
Combinatorially, the graph $|\mathcal{G}^\alpha|$ is obtained from $|\mathcal{G}|$
by adding one new edge $\alpha$ incident to $v$ and to one new vertex $v_\alpha$.
(We orient $\alpha$ so that $v_\alpha$ is its terminal vertex.)
If an oriented edge $e \in \st(v)$ supports a direction $d$ contained in $\alpha$,
then in $\mathcal{G}^\alpha$,
we make the initial vertex of $e$ equal to $v_\alpha$ instead.
Otherwise the graph $|\mathcal{G}^\alpha|$ is combinatorially identical to $|\mathcal{G}|$.
The vertex and edge groups of $\mathcal{G}^\alpha$ are equal to what they were in $\mathcal{G}$,
with $\mathcal{G}^\alpha_\alpha = \mathcal{G}^\alpha_{v_\alpha}
= \mathcal{G}_\alpha \le \mathcal{G}_v$. 
The inclusion $\mathcal{G}^\alpha_\alpha \to \mathcal{G}^\alpha_v$
is the natural inclusion $\mathcal{G}_\alpha \to \mathcal{G}_v$;
the inclusion $\mathcal{G}^\alpha_\alpha \to \mathcal{G}^\alpha_{v_\alpha}$ is the identity.
If an edge $e \in \st(v)$ supports a direction $d \in \alpha$,
choose one such $d = (g\iota_e(\mathcal{G}_e), e)$.
We have that $g \mathcal{G}_e g^{-1}$,
the stabilizer of $d$ under the action of $\mathcal{G}_v$ on $D_v$,
is a subgroup of $\mathcal{G}_\alpha$,
and we let the inclusion $\mathcal{G}^\alpha_e \to \mathcal{G}^\alpha_{v_\alpha}$
be the map $h \mapsto ghg^{-1}$.
Otherwise the oriented edge $e$ has initial vertex in $|\mathcal{G}^\alpha|$
equal to what it had in $|\mathcal{G}|$,
so we let the map $\iota_e$ remain what it was in $\mathcal{G}$.

There is a homotopy equivalence $f\colon \mathcal{G} \to \mathcal{G}^\alpha$
defined as follows.
If $e$ is an edge of $|\mathcal{G}|$ in $\st(v)$
which becomes incident to $v_\alpha$ in $\mathcal{G}^\alpha$,
the image $f(e)$ is $g^{-1}\alpha e$,
where $g$ is the element chosen above.
(If \emph{both} incident vertices to $e$ in $\mathcal{G}^\alpha$ are $v_\alpha$,
we have instead that $f(e) = g^{-1}\alpha e \bar{\alpha} g'$,
where $g'$ is the element chosen for $\bar e$ above.)
The maps $f_v$ are the identity.
Each map of the form $f_{e,e}$ is also the identity.
Each map of the form $f_{e,\alpha}$ is the map $h \mapsto ghg^{-1}$.
Assuming that $f(e) = g^{-1}\alpha e$, one checks the compatibility condition as follows
\[
  \begin{tikzcd}
    \iota_e(h) \ar[d, maps to] & &
    h \in \mathcal{G}_e \ar [ll, maps to] \ar[ld, maps to] \ar[rd, maps to]
    \ar[rr, maps to] & &
    \iota_{\bar e}(h) \ar[d, maps to] \\
    \iota_e(h) & ghg^{-1} \in \mathcal{G}^\alpha_\alpha \ar[l, maps to] \ar[r, maps to]
    & ghg^{-1} & h \in \mathcal{G}^\alpha_e \ar[l, maps to] \ar[r, maps to] & \iota_{\bar e}(h).
  \end{tikzcd}
\]
It is from here an easy exercise to verify the compatibility conditions
when $f(e) = g^{-1}\alpha e \bar \alpha g'$.
To show that this map is a homotopy equivalence,
observe that $\alpha$ is a collapsible edge;
collapsing it yields a homotopy inverse for $f$.

When $\sigma\colon \mathbb{G} \to \mathcal{G}$ is a marking,
composing with $f\colon \mathcal{G} \to \mathcal{G}^\alpha$
yields a marking $\sigma^\alpha\colon \mathbb{G}\to \mathcal{G}^\alpha$
defined as $\sigma^\alpha = f\sigma$.

The following lemma is an easy exercise.
\begin{lem}\label{atleasttwo}
  Suppose that ${G}^\alpha$ has no inessential valence-one or valence-two vertices.
  Then the sets $\alpha$ and $D_v - \alpha$ in $D_v$ contain at least two elements.\hfill\qedsymbol
\end{lem}

In fact, the converse is also true provided $\mathcal{G}$ had no inessential valence-one or valence-two vertices.

Observe that if $[\beta,\mathcal{G}_\beta]$ is compatible with $[\alpha, \mathcal{G}_\alpha]$
(and not equal to $[D_v - \alpha, \mathcal{G}_v]$),
then we may think of $[\beta, \mathcal{G}_\beta]$ as an ideal edge in $\mathcal{G}^\alpha$.
If the original ideal edges were both based at $v$,
then $[\beta, \mathcal{G}_\beta]$ will continue to be based at $v$ if it is disjoint from
$[\alpha,\mathcal{G}_\alpha]$; otherwise $[\beta,\mathcal{G}_\beta]$ is based at $v_\alpha$.

Therefore by repeatedly blowing up $\sqsubset$-maximal edges of an oriented
ideal forest $\Phi^+$, we obtain a marked graph of groups $(\mathcal{G}^{\Phi^+},\sigma^{\Phi^+})$.
Repeated application of \Cref{atleasttwo} proves the following.

\begin{lem}\label{forestatleasttwo}
  Suppose that $(\mathcal{G},\sigma)$ represents a vertex of $K(\mathbb{G})$
  and $\Phi^+$ is an oriented ideal forest
  such that each ideal edge $\alpha \in \Phi^+$ satisfies
  that $\alpha$ and $D_v - \alpha$ contain at least two elements. Then
  the marked graph of groups $(\mathcal{G}^{\Phi^+},\sigma^{\Phi^+})$
  represents a vertex of $K(\mathbb{G})$.

  Moreover, the edges of $\Phi^+$ form an oriented collapsible forest in $\mathcal{G}^{\Phi^+}$
  and collapsing them recovers $(\mathcal{G},\sigma)$.\hfill\qedsymbol
\end{lem}

\begin{lem}
  Suppose $\Psi^+$ is an oriented collapsible forest in $(\mathcal{G},\sigma)$,
  a marked graph of groups representing a vertex of $K(\mathbb{G})$.
  Then in the collapsed marked graph of groups $(\mathcal{G}', \sigma')$,
  there exists an oriented ideal forest $\Phi^+$ such that the marked graph of groups
  $(\mathcal{G}'^{\Phi^+},\sigma'^{\Phi^+})$ is equivalent to $(\mathcal{G},\sigma)$.
\end{lem}

\begin{proof}
  We show that the statement in the lemma holds when collapsing
  a single edge of $\Psi^+$ which is incident to the source vertex $v_C$.
  An easy inductive argument completes the proof.

  To prove the statement, we need an ideal edge $(\alpha, \mathcal{G}_\alpha)$
  based at $v_C$ in the collapsed graph of groups such that $(\mathcal{G}'^\alpha, \sigma'^\alpha)$
  is equivalent to $(\mathcal{G},\sigma)$.
  Indeed, take $\mathcal{G}_\alpha = \mathcal{G}_e$, the stabilizer of the collapsed edge $e$.
  In the collapsed graph of groups,
  each edge group of each edge incident to the terminal vertex of $e$ in $\mathcal{G}$
  naturally has edge group contained in $\mathcal{G}_\alpha$.
  Therefore, we choose $\alpha$ to be the $\mathcal{G}_\alpha$-orbit of the direction
  $(\mathcal{G}_e' , e')$ for each oriented edge $e'$ incident to the terminal vertex of $e$ in $\mathcal{G}$.
  If one uses these directions when marking the blown up graph of groups,
  it is clear that the resulting marked graphs of groups are equivalent.
  (The isomorphism recognizes that $\alpha$ is the collapsed edge $e$ and is otherwise ``the identity''.)
\end{proof}

Note that the collections of (unoriented) collapsible forests and (unoriented) ideal forests
are partially ordered; the former by inclusion of sets of edges
and the latter by the rule that $\Phi \le \Phi'$ if there exist orientations
on each so that $\Phi^+ \subset \Phi'^+$.

Since each collapsible forest admits an orientation,
the foregoing proves the following lemma.

\begin{lem}\label{understandinglinks}
  In $K(\mathbb{G})$, the link of a vertex represented by the marked graph of groups $(\mathcal{G},\sigma)$
  is isomorphic to the join of the posets of ideal and collapsible forests in $\mathcal{G}$,
  where ideal forests are required to satisfy the conditions of \Cref{forestatleasttwo}.
  \hfill\qedsymbol
\end{lem}

\paragraph{Links in $L$.}
Although it is useful to understand links in $K(\mathbb{G})$,
when vertex groups of $\mathbb{G}$ are infinite,
the space $K(\mathbb{G})$ is often infinite dimensional,
simply because the ideal edge construction is so flexible.

If $(\mathcal{G},\sigma)$ is a marked graph of groups
representing a vertex of $L(\mathbb{G})$,
observe that if $(\mathcal{G}',\sigma')$ is a collapse of $(\mathcal{G},\sigma)$,
then every edge $e$ of $\mathcal{G}$ whose image in $\mathcal{G}'$ is surviving in $\mathcal{G}'$
is surviving in $\mathcal{G}$, simply because the collapse relation is transitive.
Conversely, we have the following lemma.

\begin{lem}\label{collapseshelters}
  Suppose $\mathcal{G}'$ is obtained from $\mathcal{G}$ by collapsing an edge.
  If $e$ is an edge of $\mathcal{G}$ which is not collapsed in $\mathcal{G}'$,
  then any shelter for $e$ in $\mathcal{G}$ which is not an ascending loop yields a shelter for $e$
  in $\mathcal{G}'$.
\end{lem}

\begin{proof}
  Let $f$ be the collapsed edge and $\mathcal{S}$ a shelter containing $e$ in $\mathcal{G}$.
  If $f$ is contained in $\mathcal{S}$,
  a simple inspection of the cases shows that $\mathcal{S}$ remains a shelter of the same type.
  So suppose that $f$ is not an edge of $\mathcal{S}$.
  If no endpoint of $f$ is contained in $\mathcal{S}$,
  then $\mathcal{S}$ projects isomorphically to $\mathcal{G}'$.
  So suppose that $f$ meets $\mathcal{S}$ in one of its endpoints, say $v$.
  Since $f$ is collapsible, if $w$ is its other endpoint,
  we may (slightly abusing notation) say that either $\mathcal{G}_v \le \mathcal{G}_w$
  or conversely.
  If $\mathcal{G}_w \le \mathcal{G}_v$, it is again clear that $\mathcal{S}$ projects isomorphically to $\mathcal{G}'$.
  So suppose that $\mathcal{G}_v$ is properly contained in $\mathcal{G}_w$.
  Then $\mathcal{S}$ is no longer a shelter of the same type.
  However, provided that $\mathcal{S}$ is not an ascending loop shelter,
  the portion containing $e$ will be either a segment shelter
  or a basepointed loop shelter.
\end{proof}

Let us remark that (keeping the notation from the foregoing proof)
in the case where $\mathcal{S}$ is an ascending loop shelter,
it is easy to come up with examples where $e$, which had a shelter in $\mathcal{G}$,
no longer has any shelter in $\mathcal{G}'$.

\begin{cor}\label{nonascendingcollapsesfine}
  If $\mathscr{T}(\mathbb{G})$ is non-ascending
  and $(\mathcal{G},\sigma)$ represents a vertex of $L(\mathbb{G})$,
  any collapse of $(\mathcal{G},\sigma)$ in $K(\mathbb{G})$
  is again in $L(\mathbb{G})$.\hfill\qedsymbol
\end{cor}

Next we consider ideal edges.
Because the collapse relation is transitive,
if every edge of $\mathcal{G}$ is surviving,
then $\mathcal{G}^\alpha$ will have every edge surviving
provided that $\alpha$ is surviving in $\mathcal{G}^\alpha$.
Suppose that $\mathcal{S}$ is a shelter containing $\alpha$.
Then by the argument in the proof of \Cref{collapseshelters},
the image of $\mathcal{S}$ after collapsing to  $\mathcal{G}$ remains a shelter of the same type.
Since $\alpha$ is collapsible in $\mathcal{G}^\alpha$, this shelter $\mathcal{S}$ has at least two edges.
Unless $\mathcal{S}$ is a segment shelter whose initial or terminal edge is $\alpha$,
then in $\mathcal{G}$, the image of $\mathcal{S}$
contains a direction in $\alpha$ and a direction in $D_v - \alpha$.
In the contrary case, in $\mathcal{G}$ there is a direction $d \in \alpha$
contained in a segment shelter such that $\mathcal{G}_d$ is a proper subgroup of $\mathcal{G}_v$.
In either case, we say that $\alpha$ \emph{cuts} the collapsed shelter $\mathcal{S}'$ in $\mathcal{G}$.

Now, if $\alpha$ cuts some shelter $\mathcal{S}'$ in $\mathcal{G}$,
it is clear that the collection $\mathcal{S}' \cup \{\alpha\}$
will be a shelter of the same type in $\mathcal{G}^\alpha$.
That is, we have the following lemma.

\begin{lem}\label{cutsshelter}
  An ideal edge $\alpha$ is surviving in $\mathcal{G}^\alpha$ if and only if
  it cuts some shelter in $\mathcal{G}$.\hfill\qedsymbol
\end{lem}

Notice that a marked graph of groups is \emph{reduced} if and only if each shelter is a single edge.
For reduced marked graphs of groups,
we have the following easy characterizations of ideal edges yielding vertices of $L(\mathbb{G})$.

\begin{lem}\label{easycase}
  Suppose that $\alpha$ is an ideal edge in a reduced marked graph of groups $\mathcal{G}$
  such that $\alpha$ and $D_v - \alpha$ contain at least two elements.
  Then in $\mathcal{G}^\alpha$, the edge $\alpha$ is contained in a shelter which is not an ascending loop
  if and only if
  there exists a direction $d \in \alpha$ with underlying oriented edge $e$
  such that $\mathcal{G}_\alpha = \mathcal{G}_d$
  and no direction with underlying oriented edge $\bar e$ is contained in $\alpha$.
\end{lem}

\begin{proof}
  In all cases, we see immediately that $\alpha$ cuts the shelter $\{e\}$.
  For segment shelters and basepointed or flat loop shelters,
  the assumption $\mathcal{G}_\alpha = \mathcal{G}_d$
  is necessary and sufficient to ensure that $\{e, \alpha\}$ is a shelter.
\end{proof}

\begin{lem}\label{ascendingeasycase}
  Suppose that the ideal edge $\alpha$ cuts an ascending loop shelter $\mathcal{S}$
  in a marked graph of groups $(\mathcal{G},\sigma)$ corresponding to a vertex of $L(\mathbb{G})$.
  Then if $\alpha$ and $D_v - \alpha$ contain at least two elements,
  we have that $\alpha$ is surviving in $\mathcal{G}^\alpha$.
\end{lem}

It is worth remarking that we do not need $\mathcal{G}$ to be reduced in this case.

\begin{proof}
  The point is that in $\mathcal{G}^\alpha$, the collection $\mathcal{S} \cup \{\alpha\}$
  is a shelter. Indeed, because $\alpha$ cuts $\mathcal{S}$,
  we see that there are a pair of directions $d$ and $d'$,
  with one, say $\mathcal{G}_d$, satisfying $\mathcal{G}_d = \mathcal{G}_v$
  with one contained in $\alpha$ and the other in $D_v - \alpha$.
  If $d \in \alpha$, then it is clear that $\mathcal{S} \cup \{\alpha\}$
  satisfies the definition of an ascending loop shelter because both edge-to-vertex group inclusions
  in $\mathcal{G}^\alpha$ involving $\alpha$ are surjective.
  In the contrary orientation, the important edge-to-vertex group inclusions
  are $\mathcal{G}_d \to \mathcal{G}_v$, which is surjective by assumption,
  and $\mathcal{G}_{\bar{\alpha}} \to \mathcal{G}_{v_\alpha}$, which is surjective by construction.
  Thus we see that $\mathcal{S}$ is a shelter for $\alpha$.
\end{proof}

In general, the conditions in \Cref{easycase}, while necessary
for $\alpha$ to be contained in a shelter which is not an ascending loop,
are not sufficient
to show that $\alpha$ cuts a shelter in $\mathcal{G}$.

\begin{lem}\label{freeproductlemma}
  When $G$ is a free product and reduced trees in $K(\mathbb{G})$ are free splittings,
  then a marked graph of groups $\tau = (\mathcal{G},\sigma)$ has every edge surviving
  just when it satisfies the following conditions.
  \begin{enumerate}
  \item Edge groups of $\mathcal{G}$ are trivial.
  \item Each valence-one and valence-two vertex of $\mathcal{G}$ has nontrivial vertex group.
  \item If an (open) edge $e$ separates $\mathcal{G}$,
    each component of the complement contains a vertex with nontrivial vertex group.
  \end{enumerate}
\end{lem}

\begin{proof}
  The deformation space $K(\mathbb{G})$ is non-ascending,
  so each loop shelter is either basepointed or flat.
  Since edge groups in reduced marked graphs of groups are trivial,
  \Cref{easycase} shows that if an ideal edge $\alpha$
  in a reduced marked graph of groups $\mathcal{G}$ is surviving in $\mathcal{G}^\alpha$,
  then $\mathcal{G}_\alpha$ must be trivial,
  proving the necessity of the first condition.
  The second condition is true in $K(\mathbb{G})$ so must be true in $L(\mathbb{G})$.
  The third condition is necessary for the edge $e$ to be contained in a segment shelter,
  as it must when it is separating.
  All three conditions together are clearly sufficient as well.
\end{proof}

Finally we come to our main result of this section.

\begin{prop}\label{idealedgeproposition}
  Suppose that the graph of groups $\mathbb{G}$ has finite edge groups
  and that $(\mathcal{G},\sigma)$ is a graph of groups in $L(\mathbb{G})$.
  Let $(\alpha,\mathcal{G}_\alpha)$ be an ideal edge in $\mathcal{G}$
  based at a vertex $v$ of $\mathcal{G}$.
  We have that the blown-up graph groups $(\mathcal{G}^\alpha,\sigma^\alpha)$
  is also in $L$ if and only if the following conditions are satisfied.
  \begin{enumerate}
  \item The sets $\alpha$ and $D_v - \alpha$ each have at least two elements.
  \item The set of directions $d \in \alpha$ such that $\mathcal{G}_d = \mathcal{G}_\alpha$
    is nonempty, call it $D(\alpha)$.
  \item Either there must be an embedded circle $S \subset |\mathcal{G}|$ containing $v$ as a vertex
    such that
    \begin{enumerate}
    \item if $e$ and $e'$ are the (distinct) edges in $S$ incident to $v$,
      one of them, say $e$, supports a direction $d \in \alpha$,
      but no direction supported by the other, say $e'$ supports a direction in $\alpha$, and
    \item each edge group in $\mathcal{G}$ for each edge of $S$ is isomorphic
      to $\mathcal{G}_\alpha$,
    \end{enumerate}
    or if this condition does not hold,
    some direction $d \in D(\alpha)$ is contained in a segment shelter
    (this must be the case if $\alpha$ separates $\mathcal{G}^\alpha$).
    Additionally if $\mathcal{G}_\alpha = \mathcal{G}_v$
    and there is no circle $S$ as above, we require also that
    some direction in $D(D_v - \alpha)$ is contained in a segment shelter.
  \end{enumerate}
\end{prop}

Before turning to the proof, let us make a couple remarks.
Notice first that finite groups cannot properly contain themselves
(they are \emph{co-Hopfian}),
and notice as well that if $(\mathcal{G},\sigma)$ is a marked graph of groups in $L(\mathbb{G})$,
then $\mathcal{G}$ has finite edge groups.
Indeed, if $\mathbb{G}$ is not reduced,
we may pass to a reduced collapse of $\mathbb{G}$;
it determines the same deformation space.
Suppose $\mathcal{G}'$ is a reduced collapse of $\mathcal{G}$.
Any homotopy equivalence $\mathbb{G} \to \mathcal{G}'$ or vice versa
must not collapse any edge, since the graphs of groups in question are reduced.
If $\mathcal{G}'$ had an infinite edge group,
in the Bass--Serre tree of $\mathbb{G}$, we would have some infinite group
fixing two distinct points; hence stabilizing some edge, a contradiction.
Therefore every edge with infinite edge group fails to be surviving.

From these considerations, we conclude that $L$ is non-ascending.

\begin{proof}
  Observe that the first condition is necessary and sufficient for
  $(\mathcal{G}^\alpha, \sigma^\alpha)$
  to represent a vertex of $K(\mathbb{G})$.
  The second condition is clearly necessary
  by considering a shelter for $\alpha$ in $\mathcal{G}^\alpha$.
  Indeed, collapsing $\alpha$ in this shelter yields a shelter of the same type,
  from which we can see that one of the options in the third condition holds.

  What's more these conditions are also sufficient.
  If the ``embedded circle'' option holds,
  then $\alpha$ does not separate $\mathcal{G}^\alpha$,
  this circle $S$ is a flat loop shelter or a basepointed loop shelter and $\alpha$ cuts it.
  If no such circle exists,
  then the hypothesized segment shelters paste together (if $\mathcal{G}_v = \mathcal{G}_\alpha$)
  to a segment shelter that $\alpha$ cuts.
\end{proof}

%% file: setupfortheoremA.tex
\section{The proof of \texorpdfstring{\Cref{mainpositiveresult}}{Theorem A}}\label{mainpositivesection}
The purpose of this section is to prove \Cref{mainpositiveresult}.
That is, let $G = A * B * \mathbb{Z}$,
where $A$ and $B$ are nontrivial finite groups,
and let $\mathbb{G}$ be a free splitting of $G$
whose associated free factor decomposition $\mathscr{A}$ of $G$
is of the form $\{[A], [B]\}$.
We sometimes write $L(G,\mathscr{A})$ for $L(\mathbb{G})$,
or just $L$ when it is clear from context what we mean.

When $G = A_1 * \cdots * A_n * F_k$ is a free product with associated free factor decomposition
$\mathscr{A} = \{[A_1],\ldots,[A_n]\}$
and the groups $A_i$ are finite (or more generally freely indecomposable and not infinite cyclic),
the deformation space $\mathscr{T}(\mathbb{G})$ is canonical,
and $\out(G)$ acts properly discontinuously and cocompactly on $L(\mathbb{G}) = L(F,\mathscr{A})$
(when the $A_i$ are finite; more generally the action is cocompact but not properly discontinuous).

In the case $G = A*B*\mathbb{Z}$,
the complex $L$ has dimension $2$,
and there are at most $7$ combinatorial types of graphs of groups determining vertices of $L$
(there are only $5$ if $A \cong B$);
one maximal with respect to expansion,
and three each of intermediate and minimal.
They are listed in \Cref{typesfigure}.

\begin{figure}
    \begin{center}
	    \def\svgwidth{\columnwidth}
	        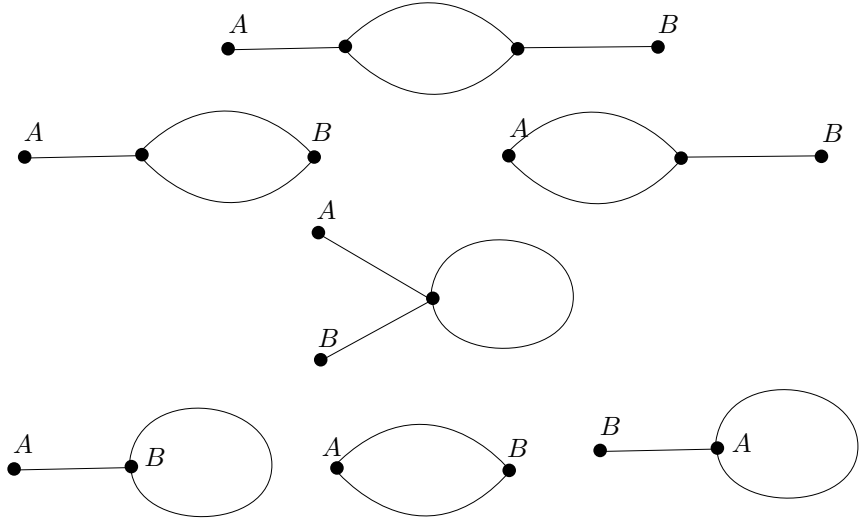
	
    \end{center}
    \caption{The combinatorial types of graphs of groups occuring in $L(F)$.}\label{typesfigure}
\end{figure}

%% file: figures/typesfigure.pdf_tex
\begingroup%
  \makeatletter%
  \providecommand\color[2][]{%
    \errmessage{(Inkscape) Color is used for the text in Inkscape, but the package 'color.sty' is not loaded}%
    \renewcommand\color[2][]{}%
  }%
  \providecommand\transparent[1]{%
    \errmessage{(Inkscape) Transparency is used (non-zero) for the text in Inkscape, but the package 'transparent.sty' is not loaded}%
    \renewcommand\transparent[1]{}%
  }%
  \providecommand\rotatebox[2]{#2}%
  \newcommand*\fsize{\dimexpr\f@size pt\relax}%
  \newcommand*\lineheight[1]{\fontsize{\fsize}{#1\fsize}\selectfont}%
  \ifx\svgwidth\undefined%
    \setlength{\unitlength}{992.12598425bp}%
    \ifx\svgscale\undefined%
      \relax%
    \else%
      \setlength{\unitlength}{\unitlength * \real{\svgscale}}%
    \fi%
  \else%
    \setlength{\unitlength}{\svgwidth}%
  \fi%
  \global\let\svgwidth\undefined%
  \global\let\svgscale\undefined%
  \makeatother%
  \begin{picture}(1,0.52857143)%
    \lineheight{1}%
    \setlength\tabcolsep{0pt}%
    \put(0,0){\includegraphics[width=\unitlength,page=1]{typesfigure.pdf}}%
    \put(0.30462159,0.46790988){\color[rgb]{0,0,0}\makebox(0,0)[lt]{\lineheight{1.25}\smash{\begin{tabular}[t]{l}$A$\end{tabular}}}}%
    \put(0.66703993,0.4679072){\color[rgb]{0,0,0}\makebox(0,0)[lt]{\lineheight{1.25}\smash{\begin{tabular}[t]{l}$B$\end{tabular}}}}%
    \put(0,0){\includegraphics[width=\unitlength,page=2]{typesfigure.pdf}}%
    \put(0.1317837,0.37667371){\color[rgb]{0,0,0}\makebox(0,0)[lt]{\lineheight{1.25}\smash{\begin{tabular}[t]{l}$A$\end{tabular}}}}%
    \put(0.37448406,0.37679687){\color[rgb]{0,0,0}\makebox(0,0)[lt]{\lineheight{1.25}\smash{\begin{tabular}[t]{l}$B$\end{tabular}}}}%
    \put(0.54136841,0.38028496){\color[rgb]{0,0,0}\makebox(0,0)[lt]{\lineheight{1.25}\smash{\begin{tabular}[t]{l}$A$\end{tabular}}}}%
    \put(0,0){\includegraphics[width=\unitlength,page=3]{typesfigure.pdf}}%
    \put(0.80560953,0.37645869){\color[rgb]{0,0,0}\makebox(0,0)[lt]{\lineheight{1.25}\smash{\begin{tabular}[t]{l}$B$\end{tabular}}}}%
    \put(0,0){\includegraphics[width=\unitlength,page=4]{typesfigure.pdf}}%
    \put(0.38349488,0.11120805){\color[rgb]{0,0,0}\makebox(0,0)[lt]{\lineheight{1.25}\smash{\begin{tabular}[t]{l}$A$\end{tabular}}}}%
    \put(0.53979502,0.10995182){\color[rgb]{0,0,0}\makebox(0,0)[lt]{\lineheight{1.25}\smash{\begin{tabular}[t]{l}$B$\end{tabular}}}}%
    \put(0,0){\includegraphics[width=\unitlength,page=5]{typesfigure.pdf}}%
    \put(0.12285698,0.11311933){\color[rgb]{0,0,0}\makebox(0,0)[lt]{\lineheight{1.25}\smash{\begin{tabular}[t]{l}$A$\end{tabular}}}}%
    \put(0.23372567,0.10213365){\color[rgb]{0,0,0}\makebox(0,0)[lt]{\lineheight{1.25}\smash{\begin{tabular}[t]{l}$B$\end{tabular}}}}%
    \put(0,0){\includegraphics[width=\unitlength,page=6]{typesfigure.pdf}}%
    \put(0.72937003,0.11431473){\color[rgb]{0,0,0}\makebox(0,0)[lt]{\lineheight{1.25}\smash{\begin{tabular}[t]{l}$A$\end{tabular}}}}%
    \put(0.61792411,0.12886246){\color[rgb]{0,0,0}\makebox(0,0)[lt]{\lineheight{1.25}\smash{\begin{tabular}[t]{l}$B$\end{tabular}}}}%
    \put(0,0){\includegraphics[width=\unitlength,page=7]{typesfigure.pdf}}%
    \put(0.37893268,0.31126002){\color[rgb]{0,0,0}\makebox(0,0)[lt]{\lineheight{1.25}\smash{\begin{tabular}[t]{l}$A$\end{tabular}}}}%
    \put(0.3798645,0.20266893){\color[rgb]{0,0,0}\makebox(0,0)[lt]{\lineheight{1.25}\smash{\begin{tabular}[t]{l}$B$\end{tabular}}}}%
    \put(0,0){\includegraphics[width=\unitlength,page=8]{typesfigure.pdf}}%
  \end{picture}%
\endgroup%

%% file: specialcase.tex
\subsection{The case of \texorpdfstring{$C_2*C_2*\mathbb{Z}$}{C2 * C2 * Z}}\label{specialcasesection}
The purpose of this subsection is to prove \Cref{specialcase} below;
the whole section is given to its proof.

Consider the following Coxeter group
\[
  W = \langle x,y,z,w \mid x^2 = y^2 = z^2 = w^2 = (xy)^2 = (xz)^4 = (yz)^4 = (xw)^4 = (yw)^4 = 1 \rangle.
\]
The defining graph of this Coxeter group has four vertices and five edges:
if $\{x,y,z,w\}$ are the vertices,
there is an edge between every pair of vertices except not $\{z,w\}$.
The edge $\{x,y\}$ has label $2$,
while the other four edges are labeled $4$.
For the CAT$(0)$ metric produced in the proof of \Cref{mainpositiveresult}
in the particular case that $A = B = C_2$,
we observed that the complex $L$ bears a striking resemblance to the Davis--Moussong complex of $W$.

Indeed, letting $G = \langle a, b, t \mid a^2 = b^2 = 1 \rangle$,
the group $\out(G)$ is generated by (the outer classes of) the following automorphisms
defined by their action on the basis $\{a,b,t\}$
\[
    \sigma \begin{dcases}
        a \mapsto b \\
        b \mapsto a \\
        t \mapsto t
    \end{dcases}
    \qquad
    \tau \begin{dcases}
        a \mapsto a \\
        b \mapsto b \\
        t \mapsto t^{-1}
    \end{dcases}
    \qquad
    L_a \begin{dcases}
        a \mapsto a \\
        b \mapsto b \\
        t \mapsto at
    \end{dcases}
    \qquad
    R_b \begin{dcases}
        a \mapsto a \\
        b \mapsto b \\
        t \mapsto tb
    \end{dcases}
    \qquad
    \chi^b_t \begin{dcases}
        a \mapsto a \\
        b \mapsto t^{-1}bt \\
        t \mapsto t.
    \end{dcases}
\]
(A reader familiar with, for example,
Gilbert's \cite{Gilbert} finite presentation of $\aut(G)$
might find it enjoyable to verify this claim.)
It is an easy exercise to verify that the the map
\[
    \Phi \begin{dcases}
        x \mapsto L_a \\
        y \mapsto R_b \\
        z \mapsto \tau \\
        w \mapsto \tau(\chi_t^b)^2
    \end{dcases}
\]
extends to a well-defined homomorphism $\Phi\colon W \to \out(G)$.
It is also not hard to show that the conjugation action of the generators of $\out(G)$
on the image of the generators of $\im(\Phi)$ yields elements of $\im(\Phi)$,
so $\im(\Phi)$ is a normal subgroup of $\out(G)$.
Since $\tau$, $L_a$ and $R_b$ are contained in $\im(\Phi)$,
the quotient $\out(G)/\im(\Phi)$ is generated by $\sigma$ and $\chi_t^b$,
which have order at most $2$ in the quotient.
Since $\sigma \chi_t^b \sigma \chi_t^b$ is inner,
the index of $\im(\Phi)$ in $\out(G)$ is at most $4$.
In fact, using the geometry of $L$, we prove the following.

\begin{thm}\label{specialcase}
  With notation as above, the homomorphism $\Phi\colon W \to \out(G)$ is injective
  and has image an index-$4$ subgroup of $\out(G)$.
  Moreover, the Davis--Moussong complex for $W$ may be identified with $L$.
\end{thm}
Since the Davis--Moussong complex for $W$ is a $2$-dimensional CAT(0) complex
with infinitely many ends,
\Cref{specialcase} implies a special case of \Cref{mainpositiveresult}.

In the particular case that $A = B = C_2$,
we will show that the complex $L$ has a very particular structure;
a piece of it is drawn in \Cref{mainfigure}.
This structure depends on two observations.

\paragraph{Intermediate graph of groups have two expansions.}
For the first,
notice that each intermediate marked graph of groups
has two expansions.
Indeed, each intermediate marked graph of groups
has a single vertex $v$ that has
either valence at least two (in fact equal) and $\mathcal{G}_v$ nontrivial
or valence at least $4$ (in fact equal) and $\mathcal{G}_v$ trivial.
It's not hard to see from \Cref{freeproductlemma}
that every ideal edge of $\mathcal{G}$ is therefore based at $v$.
In each case $D_v$ has cardinality four and each ideal edge has exactly two elements.
Up to the $\mathcal{G}_v = C_2$ action in the former case 
and complementation in the latter,
we may assume that a fixed direction $d \in D_v$
belongs to every ideal edge in $\mathcal{G}$;
there are three choices for the other direction
but in each case one is forbidden.

\paragraph{Stars of minimal graphs of groups are polygons.}
For the second,
we claim that for each minimal marked graph of groups $\tau$,
$\Star(\tau)$ in $L$ is the first barycentric subdivision of a polygon.
Indeed, this observation follows from the first:
if $\tau^\Phi$ is an expansion of $\tau$ which is maximal,
then $\Phi$ has two ideal edges
which can be collapsed independently
to produce two intermediate collapses of $\tau^\Phi$ which also collapse onto $\tau$.
Since intermediate marked graphs of groups have two expansions,
the observation follows.
We will think of the maximal marked graph of groups in $\Star(\tau)$
as the vertices of the polygon,
the intermediate marked graphs of groups as barycenters of edges,
and $\tau$ itself as the barycenter of the face.

Now consider the minimal marked graphs of groups again.
If the graph of groups $\mathcal{G}$ has an edge $e$ that forms a loop,
then every ideal edge in $\mathcal{G}$ is based at the vertex $v$ incident to that loop.
By the third condition in \Cref{freeproductlemma}
if $d$ is the other edge incident to $v$,
then every ideal edge based at $v$ contains, up to the action of $\mathcal{G}_v = C_2$,
the direction $(1,d)$.
There are four ideal edges of size two 
(If $\star$ denotes the nontrivial element of $C_2$,
any choice of direction besides $(\star,d)$ is allowed) 
and a further four of size three (fixing a choice of $(1,e)$ or $(\star,e)$,
we see that only the two directions 
with underlying oriented edge $\bar e$ may be chosen for the third).

If no edge forms a loop, then both vertices $v$ and $w$ of $\mathcal{G}$
support ideal edges.
Since in this case $v$ and $w$ have valence two,
there are as in the intermediate case only two choices of ideal edge based at $v$ or $w$
for a total of four.
In other words,
if $\tau$ is minimal,
$\Star(\tau)$ is either a quadrilateral or an octagon.
It is not hard to argue that $L$ is tiled by these squares and octagons,
which overlap in at most an edge.

\paragraph{$L$ is CAT(0).}
Give $L$ the piecewise-Euclidean metric in which each quadrilateral is a Euclidean unit square
and each octagon is a regular Euclidean octagon with unit side-length.
By \cite[Theorem 5.4]{TheBible},
since $L$ is contractible,
$L$ equipped with this metric is CAT(0) if and only if it satisfies Gromov's 
\emph{link condition} \cite[Definition 5.1]{TheBible}.
By construction, we need only check this at the vertices of each octagon and square.
Since there is one combinatorial type of maximal marked graph of groups,
there is one link to check.

Let $\tau$ be maximal.
A vertex of $\Link(\tau)$ is an edge of a polygon with vertex $\tau$;
or in other words,
a marked graph of groups obtained from $\tau$ by collapsing a single edge.
$\tau$ has two separating edges and two nonseparating edges.
The nonseparating edges cannot both be collapsed,
so these give two vertices of $\Link(\tau)$ that are not connected by an edge.
Every other pair of vertices of $\Link(\tau)$ can be collapsed simultaneously,
and $\Link(\tau)$ contains an edge of length $\frac{\pi}{2}$ if the
star of the collapsed graph of groups is a square
and length $\frac{3\pi}{4}$ if it is an octagon.
Collapsing both separating edges yields a square,
while collapsing a separating edge and a nonseparating edge yields an octagon.
The graph $\Link(\tau)$, pictured in \Cref{linkfigure} thus has the property 
that every injective loop in the graph 
has length at least $2\pi$, so by \cite[Lemma 5.6]{TheBible},
$L$ satisfies the link condition and is thus CAT(0).

\begin{figure}
    \begin{center}
	    \def\svgwidth{\columnwidth}
	        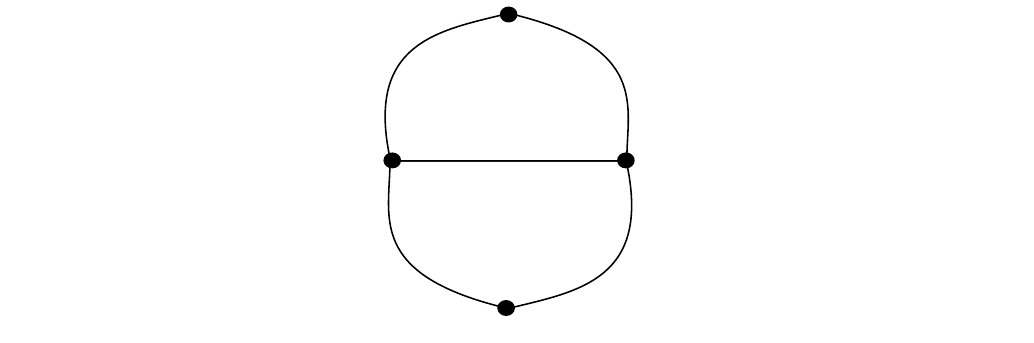

    \end{center}
    \caption{The link of a maximal marked graph of groups.}
    \label{linkfigure}
\end{figure}

\paragraph{Labelled graphs of groups.}
In \Cref{mainfigure},
we have depicted the marked graphs of groups representing vertices
and faces of polygons in $L$ by \emph{labelling} them.
Now, $G$ has the following presentation
\[
    G = \langle a, b, t \mid a^2 = b^2 = 1 \rangle.
\]
Call a triple in the $\aut(G)$-image of the set $\{a,b,t\}$ a \emph{basis} for $G$.
Under the marking, if $T$ is a maximal tree in $\tau$,
the fundamental group $\pi_1(\tau,T)$ may be identified with $G$ in the following way:
the vertex groups of $\tau$
are each generated by an order-two element of a basis for $G$,
and an edge outside of a maximal tree in $\tau$ may be identified with
the infinite-order element of that basis.
Recall that every marked graph of groups is obtained by collapsing edges
of a maximal marked graph of groups $\tau$,
which has four edges,
which we will color red, orange, pink and blue in the following way:
the pink and blue edges are nonseparating,
the red edge connects to the vertex group generated by a conjugate of $a$
and the orange edge to the vertex group generated by a conjugate of $b$.

Note that up to conjugating the basis as a whole,
we may always assume that the conjugate of $a$ chosen is actually $a$ itself.
We may furthermore assume that (up to the action of $\sigma$)
the automorphism taking our original basis to the given one
actually lies in the subgroup generated by the automorphisms
$\tau$, $L_a$, $R_b$ and $\chi_t^b$ from the introduction
since the outer classes of these automorphisms together with $\sigma$ generate $\out(G)$.

In fact, recalling the homomorphism $\Phi\colon W \to \out(G)$ from the introduction,
we claim that we may always choose the labelling automorphism
from the normal subgroup $\im(\Phi) < \out(G)$.
To see this, notice that 
the stabilizer of a maximal marked graph of groups in $L$ under the $\out(G)$-action
is isomorphic to $C_2\times C_2$
and is generated by an $\out(G)$-conjugate of the outer class of the automorphisms
$\tau\chi_t^b$ (this is the effect of swapping pink and blue edges)
and $\tau\sigma$ 
(this is the effect of ``reflecting'' the graph of groups through the midpoints of the pink and blue edges).
An arbitrary element of $\out(G)$ may be written as
\[
    w,\qquad(\chi_t^b)^{-1} w,\qquad \sigma w\qquad\text{or}\qquad (\chi_t^b)^{-1}\sigma w
\]
where $w \in \im(\Phi)$. (We have not proved that $\out(G)/\im(\Phi)$ is all of $C_2\times C_2$,
so we do not claim uniqueness here)
Since $\tau^2 = 1$, $(\chi_t^b)^{-1}\tau = \tau\chi_t^b$ and $\tau \in \im(\Phi)$,
we may introduce a pair of $\tau$ into the above equations
to see that an arbitrary labelling automorphism may be represented
by an element of the stabilizer of a fixed maximal marked graph of groups
followed by an element of $\im(\Phi)$.

\paragraph{The Davis--Moussong complex of $W$.}
The Davis--Moussong complex of $W$ is a CAT(0) polygonal complex tiled by squares and octagons.
It is built in the following way:
begin with a variant of the (typically right, but for our purposes \emph{left}) Cayley graph
of $W$ where rather than putting in a bigon between every pair of vertices that differ by a Coxeter generator,
we put a single edge.
Each Coxeter relation for $W$ determines a family of loops in the Cayley graph
of length four or eight;
fill in these loops with $2$-cells to form squares and octagons,
and give the Davis--Moussong complex the piecewise-Euclidean metric
that agrees with the metric we gave $L$ on each cell.

\paragraph{The action isomorphism.}
We turn to proving that $\Phi$ is injective
and that $L$ is isomorphic to the Davis--Moussong complex of $W$.
Observe that under $\Phi$, each of the four Coxeter generators of $W$
fixes the midpoint of a unique edge incident to the vertex represented by
our fixed maximal graph of groups.
This implies that there exists a $\Phi$-equivariant cellular map from the Davis--Moussong complex of $W$
to $L$ that is an isometry on each $2$-cell.
Put another way, we may think of vertices of $L$ as labelled by elements of $\im(\Phi)$,
and a loop in the $1$-skeleton of $L$ corresponds to a word in the generators of $\im(\Phi)$
(and thus $W$) that is trivial in $\out(G)$.
Since $L$ is simply connected,
there exists a van Kampen diagram for this word
whose interior is tiled by squares and octagons.
But each square and each octagon, which corresponds to a relator in $\im(\Phi)$,
is already a relator in $W$.
This proves that $\Phi$ is injective and the $\Phi$-equivariant cellular map
is in fact an equivariant isometry.
Since $W$ therefore acts freely on the vertices of $L$ via $\Phi$,
this also proves that $\im(\Phi)$ has index four in $\out(G)$.
This completes the proof of \Cref{specialcase}.

\begin{figure}
    \begin{center}
	    \def\svgwidth{\columnwidth}
	        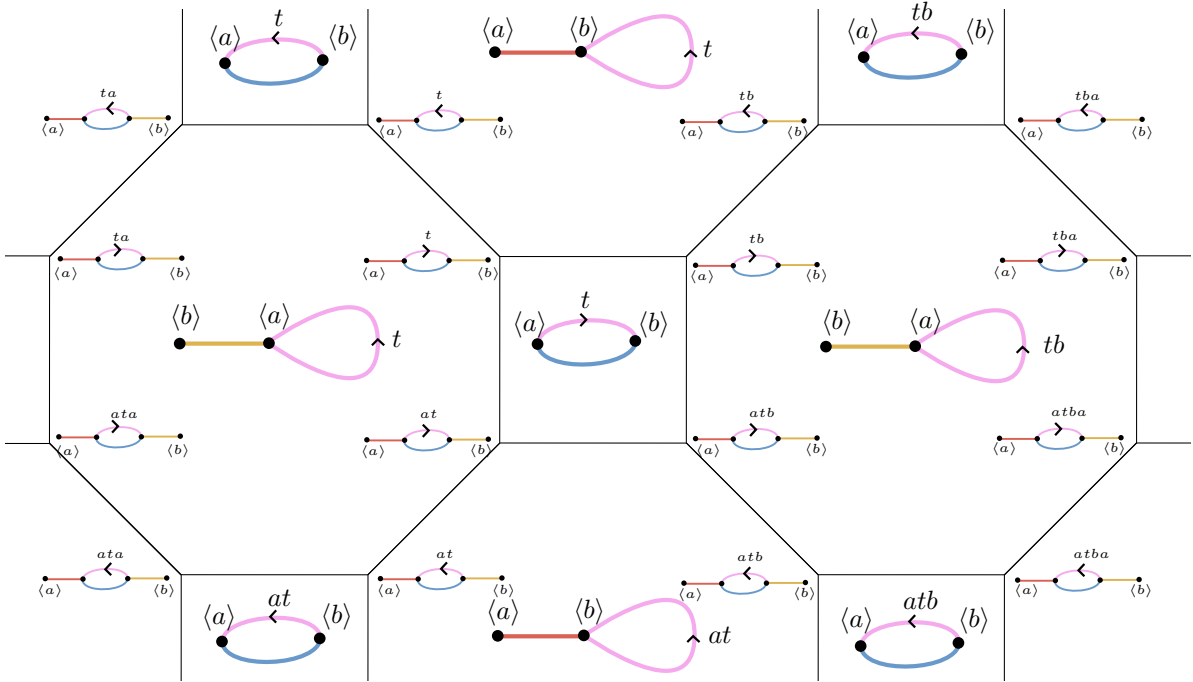
	
    \end{center}
    \caption{A portion of the complex $L$.}
    \label{mainfigure}
\end{figure}

%% file: figures/linkfigure.pdf_tex
\begingroup%
  \makeatletter%
  \providecommand\color[2][]{%
    \errmessage{(Inkscape) Color is used for the text in Inkscape, but the package 'color.sty' is not loaded}%
    \renewcommand\color[2][]{}%
  }%
  \providecommand\transparent[1]{%
    \errmessage{(Inkscape) Transparency is used (non-zero) for the text in Inkscape, but the package 'transparent.sty' is not loaded}%
    \renewcommand\transparent[1]{}%
  }%
  \providecommand\rotatebox[2]{#2}%
  \newcommand*\fsize{\dimexpr\f@size pt\relax}%
  \newcommand*\lineheight[1]{\fontsize{\fsize}{#1\fsize}\selectfont}%
  \ifx\svgwidth\undefined%
    \setlength{\unitlength}{496.06299213bp}%
    \ifx\svgscale\undefined%
      \relax%
    \else%
      \setlength{\unitlength}{\unitlength * \real{\svgscale}}%
    \fi%
  \else%
    \setlength{\unitlength}{\svgwidth}%
  \fi%
  \global\let\svgwidth\undefined%
  \global\let\svgscale\undefined%
  \makeatother%
  \begin{picture}(1,0.33142857)%
    \lineheight{1}%
    \setlength\tabcolsep{0pt}%
    \put(0,0){\includegraphics[width=\unitlength,page=1]{linkfigure.pdf}}%
    \put(0.61009853,0.24761317){\color[rgb]{0,0,0}\makebox(0,0)[lt]{\lineheight{1.25}\smash{\begin{tabular}[t]{l}$\frac{3\pi}{4}$\end{tabular}}}}%
    \put(0.38486093,0.24761317){\color[rgb]{0,0,0}\makebox(0,0)[lt]{\lineheight{1.25}\smash{\begin{tabular}[t]{l}$\frac{3\pi}{4}$\end{tabular}}}}%
    \put(0.38730918,0.10561554){\color[rgb]{0,0,0}\makebox(0,0)[lt]{\lineheight{1.25}\smash{\begin{tabular}[t]{l}$\frac{3\pi}{4}$\end{tabular}}}}%
    \put(0.6166272,0.10561554){\color[rgb]{0,0,0}\makebox(0,0)[lt]{\lineheight{1.25}\smash{\begin{tabular}[t]{l}$\frac{3\pi}{4}$\end{tabular}}}}%
    \put(0.47952606,0.18559121){\color[rgb]{0,0,0}\makebox(0,0)[lt]{\lineheight{1.25}\smash{\begin{tabular}[t]{l}$\frac{\pi}{2}$\end{tabular}}}}%
  \end{picture}%
\endgroup%

%% file: figures/mainfigure.pdf_tex
\begingroup%
  \makeatletter%
  \providecommand\color[2][]{%
    \errmessage{(Inkscape) Color is used for the text in Inkscape, but the package 'color.sty' is not loaded}%
    \renewcommand\color[2][]{}%
  }%
  \providecommand\transparent[1]{%
    \errmessage{(Inkscape) Transparency is used (non-zero) for the text in Inkscape, but the package 'transparent.sty' is not loaded}%
    \renewcommand\transparent[1]{}%
  }%
  \providecommand\rotatebox[2]{#2}%
  \newcommand*\fsize{\dimexpr\f@size pt\relax}%
  \newcommand*\lineheight[1]{\fontsize{\fsize}{#1\fsize}\selectfont}%
  \ifx\svgwidth\undefined%
    \setlength{\unitlength}{850.39370079bp}%
    \ifx\svgscale\undefined%
      \relax%
    \else%
      \setlength{\unitlength}{\unitlength * \real{\svgscale}}%
    \fi%
  \else%
    \setlength{\unitlength}{\svgwidth}%
  \fi%
  \global\let\svgwidth\undefined%
  \global\let\svgscale\undefined%
  \makeatother%
  \begin{picture}(1,0.56666667)%
    \lineheight{1}%
    \setlength\tabcolsep{0pt}%
    \put(0,0){\includegraphics[width=\unitlength,page=1]{mainfigure.pdf}}%
    \put(0.48536723,0.31410934){\color[rgb]{0,0,0}\makebox(0,0)[lt]{\lineheight{1.25}\smash{\begin{tabular}[t]{l}$t$\end{tabular}}}}%
    \put(0.42791478,0.29518138){\color[rgb]{0,0,0}\makebox(0,0)[lt]{\lineheight{1.25}\smash{\begin{tabular}[t]{l}$\langle a \rangle$\end{tabular}}}}%
    \put(0.53364319,0.29544538){\color[rgb]{0,0,0}\makebox(0,0)[lt]{\lineheight{1.25}\smash{\begin{tabular}[t]{l}$\langle b \rangle$\end{tabular}}}}%
    \put(0.22668498,0.55244874){\color[rgb]{0,0,0}\makebox(0,0)[lt]{\lineheight{1.25}\smash{\begin{tabular}[t]{l}$t$\end{tabular}}}}%
    \put(0.17099627,0.53704853){\color[rgb]{0,0,0}\makebox(0,0)[lt]{\lineheight{1.25}\smash{\begin{tabular}[t]{l}$\langle a \rangle$\end{tabular}}}}%
    \put(0.27143316,0.53731254){\color[rgb]{0,0,0}\makebox(0,0)[lt]{\lineheight{1.25}\smash{\begin{tabular}[t]{l}$\langle b \rangle$\end{tabular}}}}%
    \put(0.76478012,0.55785674){\color[rgb]{0,0,0}\makebox(0,0)[lt]{\lineheight{1.25}\smash{\begin{tabular}[t]{l}$tb$\end{tabular}}}}%
    \put(0.70909154,0.54245652){\color[rgb]{0,0,0}\makebox(0,0)[lt]{\lineheight{1.25}\smash{\begin{tabular}[t]{l}$\langle a \rangle$\end{tabular}}}}%
    \put(0.80952835,0.54272052){\color[rgb]{0,0,0}\makebox(0,0)[lt]{\lineheight{1.25}\smash{\begin{tabular}[t]{l}$\langle b \rangle$\end{tabular}}}}%
    \put(0.75760829,0.06044724){\color[rgb]{0,0,0}\makebox(0,0)[lt]{\lineheight{1.25}\smash{\begin{tabular}[t]{l}$atb$\end{tabular}}}}%
    \put(0.70368357,0.045047){\color[rgb]{0,0,0}\makebox(0,0)[lt]{\lineheight{1.25}\smash{\begin{tabular}[t]{l}$\langle a \rangle$\end{tabular}}}}%
    \put(0.80412038,0.04531097){\color[rgb]{0,0,0}\makebox(0,0)[lt]{\lineheight{1.25}\smash{\begin{tabular}[t]{l}$\langle b \rangle$\end{tabular}}}}%
    \put(0.22127703,0.06315125){\color[rgb]{0,0,0}\makebox(0,0)[lt]{\lineheight{1.25}\smash{\begin{tabular}[t]{l}$at$\end{tabular}}}}%
    \put(0.16558832,0.04775101){\color[rgb]{0,0,0}\makebox(0,0)[lt]{\lineheight{1.25}\smash{\begin{tabular}[t]{l}$\langle a \rangle$\end{tabular}}}}%
    \put(0.26602521,0.04801498){\color[rgb]{0,0,0}\makebox(0,0)[lt]{\lineheight{1.25}\smash{\begin{tabular}[t]{l}$\langle b \rangle$\end{tabular}}}}%
    \put(0.40107524,0.54312566){\color[rgb]{0,0,0}\makebox(0,0)[lt]{\lineheight{1.25}\smash{\begin{tabular}[t]{l}$\langle a \rangle$\end{tabular}}}}%
    \put(0.47474396,0.544027){\color[rgb]{0,0,0}\makebox(0,0)[lt]{\lineheight{1.25}\smash{\begin{tabular}[t]{l}$\langle b \rangle$\end{tabular}}}}%
    \put(0.58869288,0.52695522){\color[rgb]{0,0,0}\makebox(0,0)[lt]{\lineheight{1.25}\smash{\begin{tabular}[t]{l}$t$\end{tabular}}}}%
    \put(0.4055366,0.05144185){\color[rgb]{0,0,0}\makebox(0,0)[lt]{\lineheight{1.25}\smash{\begin{tabular}[t]{l}$\langle a \rangle$\end{tabular}}}}%
    \put(0.47920532,0.05234319){\color[rgb]{0,0,0}\makebox(0,0)[lt]{\lineheight{1.25}\smash{\begin{tabular}[t]{l}$\langle b \rangle$\end{tabular}}}}%
    \put(0.59315425,0.03174367){\color[rgb]{0,0,0}\makebox(0,0)[lt]{\lineheight{1.25}\smash{\begin{tabular}[t]{l}$at$\end{tabular}}}}%
    \put(0.13874237,0.30111084){\color[rgb]{0,0,0}\makebox(0,0)[lt]{\lineheight{1.25}\smash{\begin{tabular}[t]{l}$\langle b \rangle$\end{tabular}}}}%
    \put(0.21241109,0.30201218){\color[rgb]{0,0,0}\makebox(0,0)[lt]{\lineheight{1.25}\smash{\begin{tabular}[t]{l}$\langle a \rangle$\end{tabular}}}}%
    \put(0.32636002,0.28141261){\color[rgb]{0,0,0}\makebox(0,0)[lt]{\lineheight{1.25}\smash{\begin{tabular}[t]{l}$t$\end{tabular}}}}%
    \put(0.68855479,0.29660418){\color[rgb]{0,0,0}\makebox(0,0)[lt]{\lineheight{1.25}\smash{\begin{tabular}[t]{l}$\langle b \rangle$\end{tabular}}}}%
    \put(0.76222351,0.29750552){\color[rgb]{0,0,0}\makebox(0,0)[lt]{\lineheight{1.25}\smash{\begin{tabular}[t]{l}$\langle a \rangle$\end{tabular}}}}%
    \put(0.87617249,0.27690595){\color[rgb]{0,0,0}\makebox(0,0)[lt]{\lineheight{1.25}\smash{\begin{tabular}[t]{l}$tb$\end{tabular}}}}%
    \put(0.29988691,0.34028674){\color[rgb]{0,0,0}\makebox(0,0)[lt]{\lineheight{1.25}\smash{\begin{tabular}[t]{l}\tiny$\langle a \rangle$\end{tabular}}}}%
    \put(0.39709137,0.34055074){\color[rgb]{0,0,0}\makebox(0,0)[lt]{\lineheight{1.25}\smash{\begin{tabular}[t]{l}\tiny$\langle b \rangle$\end{tabular}}}}%
    \put(0.35361766,0.37269591){\color[rgb]{0,0,0}\makebox(0,0)[lt]{\lineheight{1.25}\smash{\begin{tabular}[t]{l}\tiny$t$\end{tabular}}}}%
    \put(0.31426943,0.46061454){\color[rgb]{0,0,0}\makebox(0,0)[lt]{\lineheight{1.25}\smash{\begin{tabular}[t]{l}\tiny$\langle a \rangle$\end{tabular}}}}%
    \put(0.40971001,0.46087854){\color[rgb]{0,0,0}\makebox(0,0)[lt]{\lineheight{1.25}\smash{\begin{tabular}[t]{l}\tiny$\langle b \rangle$\end{tabular}}}}%
    \put(0.36270851,0.49125982){\color[rgb]{0,0,0}\makebox(0,0)[lt]{\lineheight{1.25}\smash{\begin{tabular}[t]{l}\tiny$t$\end{tabular}}}}%
    \put(0.56848378,0.45971321){\color[rgb]{0,0,0}\makebox(0,0)[lt]{\lineheight{1.25}\smash{\begin{tabular}[t]{l}\tiny$\langle a \rangle$\end{tabular}}}}%
    \put(0.65863281,0.45997721){\color[rgb]{0,0,0}\makebox(0,0)[lt]{\lineheight{1.25}\smash{\begin{tabular}[t]{l}\tiny$\langle b \rangle$\end{tabular}}}}%
    \put(0.61868678,0.49035849){\color[rgb]{0,0,0}\makebox(0,0)[lt]{\lineheight{1.25}\smash{\begin{tabular}[t]{l}\tiny$tb$\end{tabular}}}}%
    \put(0.57704646,0.33848406){\color[rgb]{0,0,0}\makebox(0,0)[lt]{\lineheight{1.25}\smash{\begin{tabular}[t]{l}\tiny$\langle a \rangle$\end{tabular}}}}%
    \put(0.67425095,0.33874807){\color[rgb]{0,0,0}\makebox(0,0)[lt]{\lineheight{1.25}\smash{\begin{tabular}[t]{l}\tiny$\langle b \rangle$\end{tabular}}}}%
    \put(0.62724945,0.36912936){\color[rgb]{0,0,0}\makebox(0,0)[lt]{\lineheight{1.25}\smash{\begin{tabular}[t]{l}\tiny$tb$\end{tabular}}}}%
    \put(0.58233805,0.19156698){\color[rgb]{0,0,0}\makebox(0,0)[lt]{\lineheight{1.25}\smash{\begin{tabular}[t]{l}\tiny$\langle a \rangle$\end{tabular}}}}%
    \put(0.67425095,0.19183098){\color[rgb]{0,0,0}\makebox(0,0)[lt]{\lineheight{1.25}\smash{\begin{tabular}[t]{l}\tiny$\langle b \rangle$\end{tabular}}}}%
    \put(0.62724945,0.22221227){\color[rgb]{0,0,0}\makebox(0,0)[lt]{\lineheight{1.25}\smash{\begin{tabular}[t]{l}\tiny$atb$\end{tabular}}}}%
    \put(0.29943624,0.19066564){\color[rgb]{0,0,0}\makebox(0,0)[lt]{\lineheight{1.25}\smash{\begin{tabular}[t]{l}\tiny$\langle a \rangle$\end{tabular}}}}%
    \put(0.39134905,0.19092964){\color[rgb]{0,0,0}\makebox(0,0)[lt]{\lineheight{1.25}\smash{\begin{tabular}[t]{l}\tiny$\langle b \rangle$\end{tabular}}}}%
    \put(0.34963923,0.22131093){\color[rgb]{0,0,0}\makebox(0,0)[lt]{\lineheight{1.25}\smash{\begin{tabular}[t]{l}\tiny$at$\end{tabular}}}}%
    \put(0.56713181,0.06988714){\color[rgb]{0,0,0}\makebox(0,0)[lt]{\lineheight{1.25}\smash{\begin{tabular}[t]{l}\tiny$\langle a \rangle$\end{tabular}}}}%
    \put(0.66433629,0.07015116){\color[rgb]{0,0,0}\makebox(0,0)[lt]{\lineheight{1.25}\smash{\begin{tabular}[t]{l}\tiny$\langle b \rangle$\end{tabular}}}}%
    \put(0.6173348,0.10053243){\color[rgb]{0,0,0}\makebox(0,0)[lt]{\lineheight{1.25}\smash{\begin{tabular}[t]{l}\tiny$atb$\end{tabular}}}}%
    \put(0.83212336,0.19246831){\color[rgb]{0,0,0}\makebox(0,0)[lt]{\lineheight{1.25}\smash{\begin{tabular}[t]{l}\tiny$\langle a \rangle$\end{tabular}}}}%
    \put(0.92580007,0.19273229){\color[rgb]{0,0,0}\makebox(0,0)[lt]{\lineheight{1.25}\smash{\begin{tabular}[t]{l}\tiny$\langle b \rangle$\end{tabular}}}}%
    \put(0.88232635,0.22311361){\color[rgb]{0,0,0}\makebox(0,0)[lt]{\lineheight{1.25}\smash{\begin{tabular}[t]{l}\tiny$atba$\end{tabular}}}}%
    \put(0.85105133,0.0725912){\color[rgb]{0,0,0}\makebox(0,0)[lt]{\lineheight{1.25}\smash{\begin{tabular}[t]{l}\tiny$\langle a \rangle$\end{tabular}}}}%
    \put(0.94825592,0.07285517){\color[rgb]{0,0,0}\makebox(0,0)[lt]{\lineheight{1.25}\smash{\begin{tabular}[t]{l}\tiny$\langle b \rangle$\end{tabular}}}}%
    \put(0.90125437,0.10323649){\color[rgb]{0,0,0}\makebox(0,0)[lt]{\lineheight{1.25}\smash{\begin{tabular}[t]{l}\tiny$atba$\end{tabular}}}}%
    \put(0.8330247,0.3411881){\color[rgb]{0,0,0}\makebox(0,0)[lt]{\lineheight{1.25}\smash{\begin{tabular}[t]{l}\tiny$\langle a \rangle$\end{tabular}}}}%
    \put(0.93022929,0.34145208){\color[rgb]{0,0,0}\makebox(0,0)[lt]{\lineheight{1.25}\smash{\begin{tabular}[t]{l}\tiny$\langle b \rangle$\end{tabular}}}}%
    \put(0.88322774,0.37183339){\color[rgb]{0,0,0}\makebox(0,0)[lt]{\lineheight{1.25}\smash{\begin{tabular}[t]{l}\tiny$tba$\end{tabular}}}}%
    \put(0.85548045,0.45926258){\color[rgb]{0,0,0}\makebox(0,0)[lt]{\lineheight{1.25}\smash{\begin{tabular}[t]{l}\tiny$\langle a \rangle$\end{tabular}}}}%
    \put(0.94915731,0.45952656){\color[rgb]{0,0,0}\makebox(0,0)[lt]{\lineheight{1.25}\smash{\begin{tabular}[t]{l}\tiny$\langle b \rangle$\end{tabular}}}}%
    \put(0.90215576,0.48990787){\color[rgb]{0,0,0}\makebox(0,0)[lt]{\lineheight{1.25}\smash{\begin{tabular}[t]{l}\tiny$tba$\end{tabular}}}}%
    \put(0.31295621,0.07259115){\color[rgb]{0,0,0}\makebox(0,0)[lt]{\lineheight{1.25}\smash{\begin{tabular}[t]{l}\tiny$\langle a \rangle$\end{tabular}}}}%
    \put(0.4101607,0.07285517){\color[rgb]{0,0,0}\makebox(0,0)[lt]{\lineheight{1.25}\smash{\begin{tabular}[t]{l}\tiny$\langle b \rangle$\end{tabular}}}}%
    \put(0.36315921,0.10323644){\color[rgb]{0,0,0}\makebox(0,0)[lt]{\lineheight{1.25}\smash{\begin{tabular}[t]{l}\tiny$at$\end{tabular}}}}%
    \put(0.02633268,0.07259115){\color[rgb]{0,0,0}\makebox(0,0)[lt]{\lineheight{1.25}\smash{\begin{tabular}[t]{l}\tiny$\langle a \rangle$\end{tabular}}}}%
    \put(0.12353717,0.07285517){\color[rgb]{0,0,0}\makebox(0,0)[lt]{\lineheight{1.25}\smash{\begin{tabular}[t]{l}\tiny$\langle b \rangle$\end{tabular}}}}%
    \put(0.07653567,0.10323644){\color[rgb]{0,0,0}\makebox(0,0)[lt]{\lineheight{1.25}\smash{\begin{tabular}[t]{l}\tiny$ata$\end{tabular}}}}%
    \put(0.04157777,0.19066559){\color[rgb]{0,0,0}\makebox(0,0)[lt]{\lineheight{1.25}\smash{\begin{tabular}[t]{l}\tiny$\langle a \rangle$\end{tabular}}}}%
    \put(0.13525448,0.19092964){\color[rgb]{0,0,0}\makebox(0,0)[lt]{\lineheight{1.25}\smash{\begin{tabular}[t]{l}\tiny$\langle b \rangle$\end{tabular}}}}%
    \put(0.08825298,0.22131091){\color[rgb]{0,0,0}\makebox(0,0)[lt]{\lineheight{1.25}\smash{\begin{tabular}[t]{l}\tiny$ata$\end{tabular}}}}%
    \put(0.03985265,0.34028669){\color[rgb]{0,0,0}\makebox(0,0)[lt]{\lineheight{1.25}\smash{\begin{tabular}[t]{l}\tiny$\langle a \rangle$\end{tabular}}}}%
    \put(0.13705714,0.34055074){\color[rgb]{0,0,0}\makebox(0,0)[lt]{\lineheight{1.25}\smash{\begin{tabular}[t]{l}\tiny$\langle b \rangle$\end{tabular}}}}%
    \put(0.09005564,0.37093199){\color[rgb]{0,0,0}\makebox(0,0)[lt]{\lineheight{1.25}\smash{\begin{tabular}[t]{l}\tiny$ta$\end{tabular}}}}%
    \put(0.02903667,0.46196649){\color[rgb]{0,0,0}\makebox(0,0)[lt]{\lineheight{1.25}\smash{\begin{tabular}[t]{l}\tiny$\langle a \rangle$\end{tabular}}}}%
    \put(0.1191856,0.46223054){\color[rgb]{0,0,0}\makebox(0,0)[lt]{\lineheight{1.25}\smash{\begin{tabular}[t]{l}\tiny$\langle b \rangle$\end{tabular}}}}%
    \put(0.07923966,0.4926118){\color[rgb]{0,0,0}\makebox(0,0)[lt]{\lineheight{1.25}\smash{\begin{tabular}[t]{l}\tiny$ta$\end{tabular}}}}%
  \end{picture}%
\endgroup%

%% file: maintheorem.tex
\subsection{The general case}\label{maintheoremsection}
The purpose of this subsection is to complete the proof of \Cref{mainpositiveresult};
the entire section is devoted to its proof.

Recall from \Cref{typesfigure} that there are seven combinatorial types of marked graphs of groups.
In the general case it is not true that stars of minimal marked graphs of groups are polygons,
but they may be, in a certain sense, thought of as being made by gluing a family 
of barycentrically subdivided squares
or a family of barycentrically subdivided octagons together.
We proceed through the combinatorial types.

\paragraph{Maximal links.}
If $\tau$ is maximal with respect to expansion,
then any edge of $\tau$ may be collapsed,
and a pair of edges may be collapsed if and only if they are not both nonseparating.
Thus the combinatorial type of $\Link(\tau)$ is the same as in the special case $A = B = C_2$.
The piecewise-Euclidean metric we give $L$ will make $\Link(\tau)$ isometric to the special case,
so these links will satisfy Gromov's link condition.

\paragraph{Intermediate links.}
There are three types of intermediate marked graphs of groups.
The type which has one loop edge has the same link as in the special case,
(a topological circle given the combinatorial structure of a complete bipartite graph on $2+2$ vertices)
because every ideal edge is based at a vertex with trivial vertex group, just as in the special case.
Also as in the special case, we will treat
vertices corresponding to this type of marked graph of groups as being the barycenter of an edge,
so each edge in $\Link(\tau)$ in this case will have length $\frac{\pi}{2}$;
these links thus satisfy Gromov's link condition.

In the two other cases, every ideal edge is based at a vertex with valence two
and vertex group $A$ or $B$.
A simple calculation shows that there are thus $|A|$ or $|B|$ ideal edges based at this vertex, respectively.
Every edge of these graphs of groups may be collapsed (individually),
so it follows from \Cref{understandinglinks}
and our understanding of links in $L$
that $\Star(\tau)$ is the cone on a complete bipartite graph with $3 + |A|$ or $3 + |B|$ vertices respectively.
The piecewise-Euclidean metric we give $L$ will make each edge in $\Link(\tau)$ have length $\frac{\pi}{2}$;
these links thus satisfy Gromov's link condition.

\paragraph{Minimal links.}
There are three types of minimal marked graphs of groups.
Consider the first the type with two nonseparating edges.
Every ideal forest in such a marked graph of groups $\tau$ has one ideal edge based at
a vertex with vertex group $A$ and one at the other vertex;
every such pair forms an ideal forest, so $\Link(\tau)$ is a complete bipartite graph with $|A| + |B|$ vertices.
The piecewise-Euclidean metric we give $L$ will make each edge in $\Link(\tau)$ have length $\frac{\pi}{2}$;
these links thus satisfy Gromov's link condition.

In the final two cases, every ideal edge is based at a vertex with vertex group $A$ or $B$.
In each case, a similar argument as in the previous case shows that there are $2|A|$ or $2|B|$
ideal edges of size two and $|A|^2$ or $|B|^2$ ideal edges of size three respectively.
A pair of ideal edges forms an ideal forest if and only if one is contained in the other,
so each ideal edge of size two is compatible with $|A|$ or $|B|$ ideal edges of size three respectively,
while each ideal edge of size three is compatible with two ideal edges of size two.
Explicitly, label the oriented edges incident to the relevant vertex $v$ as
$d$, $e$ and $\bar e$.
Up to the $\mathcal{G}_v$-action we may assume that $(1,d)$ is contained in every ideal edge.
The ideal edges of size two are thus in bijection with the set
$\mathcal{G}_v \times \{e\} \sqcup \mathcal{G}_v\times \{\bar e\}$
and the ideal edges of size three are in bijection with
$(\mathcal{G}_v \times\{e\})\times (\mathcal{G}_v \times \{\bar e\})$.
An injective loop in $\Link(\tau)$ of minimal length thus visits 
a sequence of size-three ideal edges of the form
$((g,e),(g',\bar e))$, $((h,e),(g',\bar e))$, $((h,e),(h',\bar e))$, $((g,e),(h',\bar e))$
and thus has length sixteen (four size-three ideal edges, four size-two, and eight unions of two).
The piecewise-Euclidean metric we give $L$ will make each edge in $\Link(\tau)$
have length $\frac{\pi}{8}$;
these links thus satisfy Gromov's link condition.

\paragraph{$L$ is CAT(0).}
Consider the pieces in \Cref{piecesfig}.
By the preceding discussion, we see that $L$ is tiled by cells isomorphic to these pieces
in the sense that every $2$-simplex of $L$ is contained in a unique such piece.
The piece with five edges contains four $2$-simplices of $L$;
the piece with four contains two.
Give $L$ the piecewise-Euclidean metric in which the piece with five edges
is a quarter of a regular octagon of side-length one
and the piece with four edges is a quarter of a square of side-length one.
By the preceding discussion we see that with this metric, $L$ satisfies Gromov's link condition
and is thus CAT(0).

\begin{figure}
    \begin{center}
	    \def\svgwidth{\columnwidth}
	        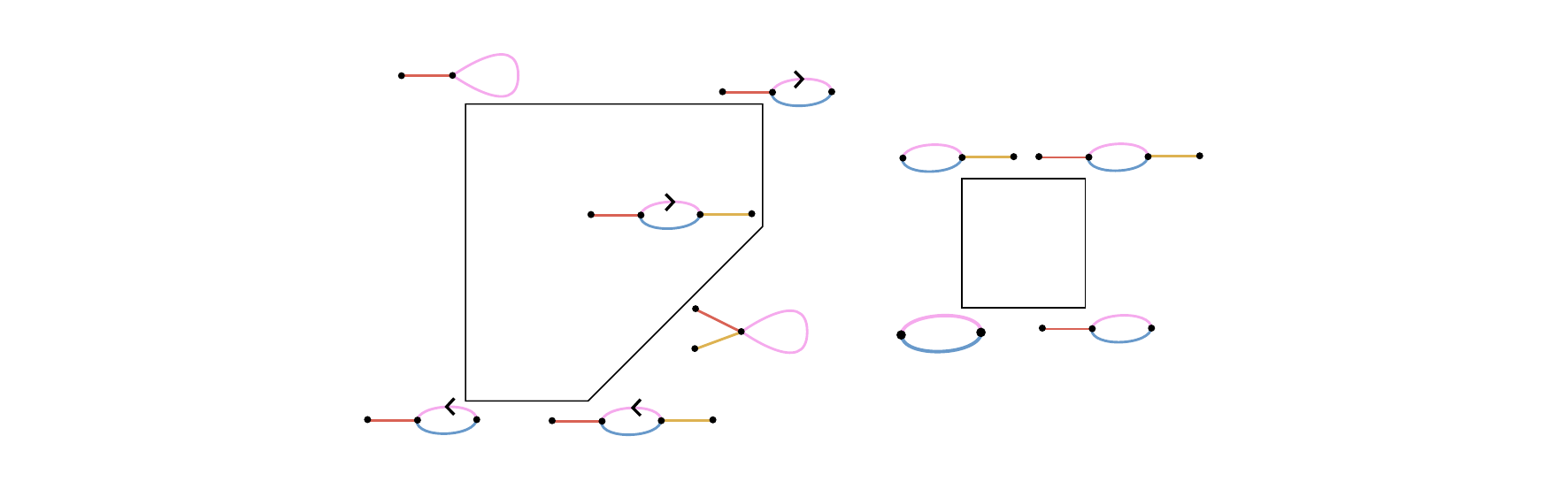
	
    \end{center}
    \caption{The pieces the complex $L$ is tiled by.}
    \label{piecesfig}
\end{figure}

\paragraph{$L$ has infinitely many ends.}
Just as the Davis--Moussong complex of $W$ had infinitely many ends,
so too does $L$ for general $A$ and $B$.
To see this, consider $\Star(\tau)$,
where $\tau$ is a minimal marked graph of groups with two nonseparating edges.
We claim that $\Star(\tau)$ separates $L$ into at least two components with noncompact closure;
since $\out(F)$ is not virtually cyclic,
(we proved in the introduction that it contains
$\mathbb{Z} \oplus \mathbb{Z}$)
this proves that $L$ has infinitely many ends.
To see that $\Star(\tau)$ separates,
consider the three types of vertices on the boundary of $\Star(\tau)$.
One has link given in \Cref{linkfigure},
in which the subset of that link contained in $\Star(\tau)$
is the closed edge of length $\frac{\pi}{2}$,
and the other two types have links that are complete bipartite on $3 + |A|$ or $3 + |B|$ vertices,
in which the subset of that link contained in $\Star(\tau)$
is the $1$-neighborhood of one vertex from the group of $3$
(thus it contains either $|A| + 1$ or $|B| + 1$ vertices).
In either case, we see that the subset of the link contained in $\Star(\tau)$
separates the link of the corresponding vertex into two pieces:
explicitly in the coloring scheme from \Cref{specialcase},
marked graphs of groups in the link without a pink edge are separated from those without a blue edge.
It follows that $\Star(\tau)$ separates a small tubular neighborhood of $\Star(\tau)$
into two pieces.
Because $L$ is CAT(0) and thus uniquely geodesic,
it follows that $\Star(\tau)$ separates $L$ into two pieces.
These pieces have noncompact closure:
indeed, each contains a Euclidean plane tiled by squares and octagons minus one square.
(This plane is in fact stabilized by a $\mathbb{Z} \oplus \mathbb{Z}$ subgroup of $\out(G)$.)

Let us conclude by remarking that while the Davis--Moussong complex for $W$ has isolated flats,
the same does not appear to be true for $L$ in general.
Indeed, one can build Euclidean planes tiled by squares and octagons
that have unbounded intersection.
(Let us remark that in \cite{MyOneEnded}, we have described algebraically that $\out(\mathbb{G})$
is virtually a free product of two direct products of free groups;
in general these free groups are not infinite cyclic, and we have an algebraic obstruction
to $L$ having isolated flats.)

%% file: figures/pieces.pdf_tex
\begingroup%
  \makeatletter%
  \providecommand\color[2][]{%
    \errmessage{(Inkscape) Color is used for the text in Inkscape, but the package 'color.sty' is not loaded}%
    \renewcommand\color[2][]{}%
  }%
  \providecommand\transparent[1]{%
    \errmessage{(Inkscape) Transparency is used (non-zero) for the text in Inkscape, but the package 'transparent.sty' is not loaded}%
    \renewcommand\transparent[1]{}%
  }%
  \providecommand\rotatebox[2]{#2}%
  \newcommand*\fsize{\dimexpr\f@size pt\relax}%
  \newcommand*\lineheight[1]{\fontsize{\fsize}{#1\fsize}\selectfont}%
  \ifx\svgwidth\undefined%
    \setlength{\unitlength}{850.39370079bp}%
    \ifx\svgscale\undefined%
      \relax%
    \else%
      \setlength{\unitlength}{\unitlength * \real{\svgscale}}%
    \fi%
  \else%
    \setlength{\unitlength}{\svgwidth}%
  \fi%
  \global\let\svgwidth\undefined%
  \global\let\svgscale\undefined%
  \makeatother%
  \begin{picture}(1,0.31666667)%
    \lineheight{1}%
    \setlength\tabcolsep{0pt}%
    \put(0,0){\includegraphics[width=\unitlength,page=1]{pieces.pdf}}%
  \end{picture}%
\endgroup%

%% file: mainnegativetheorem.tex
\section{The proof of \texorpdfstring{\Cref{mainnegativeresult}}{Theorem B}}\label{mainnegativesection}
The purpose of this section is to prove \Cref{mainnegativeresult}.
Because trees (contractible simplicial complexes of dimension one)
admit piecewise-Euclidean CAT$(0)$ metrics
and in view of \Cref{mainpositiveresult},
it suffices to show that when $L(G,\mathscr{A})$
does not admit an $\out(G,\mathscr{A})$-equivariant piecewise-Euclidean or piecewise-hyperbolic CAT$(0)$ metric
when it has dimension at least two and $L(G,\mathscr{A})$ is not the deformation space
of a free product decomposition of the form $G = A_1 * A_2 * \mathbb{Z}$.
Notice that \Cref{mainpositiveresult} is for $A_1$ and $A_2$ nontrivial \emph{finite} groups.
This is necessary only for $\out(G,\mathscr{A})$ to act geometrically on $L(G,\mathscr{A})$
equipped with the CAT$(0)$ metric we produced in \Cref{mainpositivesection}; although the complex will no longer be locally finite
if either $A_1$ or $A_2$ is an infinite group,
the metric cell structure produced in that section remains CAT$(0)$.

It is useful to recall precisely what we mean by ``piecewise-Euclidean'' or ``piecewise-hyperbolic''.
The complex $L(G,\mathscr{A})$ comes equipped with a simplicial structure.
We mean to assign, in an $\out(G,\mathscr{A})$-equivariant way,
a Euclidean or hyperbolic metric to each simplex
so that its faces are totally geodesic (i.e.\ convex),
glue these metric simplices by isometries of their faces,
and give the entire space the path metric.
Because $\out(G,\mathscr{A})$ acts cocompactly,
any such metric will have finitely many isometry types of simplices,
and thus will be CAT$(0)$ if it satisfies Gromov's
\emph{link condition}~\cite{TheBible},
which stipulates that links of simplices
equipped with the natural \emph{angle metric},
are CAT$(1)$ simplicial complexes.

In a CAT$(1)$ metric space, every geodesic loop must have length at least $2\pi$.
In particular, if the metric space is a metric \emph{graph,}
every edge path loop must have length at least $2\pi$,
while to be piecewise-Euclidean each $2$-simplex must have angle sum $\pi$.
Hyperbolic geometry places restrictions on the \emph{lengths}
of edges of geodesic triangles based on their angle sum,
but each hyperbolic triangle must have angle sum strictly less than $\pi$.
Our proof of \Cref{mainnegativeresult} puts these two conditions at odds with each other.

Explicitly, under our assumptions,
we find $2$-simplices in $L(G,\mathscr{A})$, two of whose angles must have size at least $\frac{\pi}{2}$
in any piecewise-Euclidean or piecewise-hyperbolic metric, a contradiction.
The $2$-simplices in question are pictured in \Cref{badtri1,badtri2,badtri3}.
When $L(G,\mathscr{A})$ has dimension at least two but is not a free product decomposition of the form
$G = A_1 * A_2 * \mathbb{Z}$, there are three base cases according to the rank of the free part:
if that rank is zero, the dimension of $L$ is $n - 2$, so we must have $n \ge 4$.
If the rank is one, the dimension of $L$ is $n$, but the case $n = 2$ is $A_1 * A_2 * \mathbb{Z}$,
so we have $n \ge 3$.
If the rank, $k$, is at least two, the dimension is either $2k + n - 2$ or $2k + n - 3$
according to the size of $n$, so we see that either $n \ge 1$ or $k \ge 3$.
(The case $2k + n - 2$ happens when $n \ge 2$.
In the particular case when $n \le 1$, the dimension is instead $2k + n - 3$.
The reason for the discrepancy is that the existence of possible segment shelters when $n \ge 2$
allows for graphs of groups in $L(G, \mathscr{A})$ with more edges.)
In view of Bridson's result~\cite{Bridson}, we may and shall assume that $n \ge 1$.
This is why there are three figures; each one illustrates one of these cases.

Let us explain further: each $2$-simplex (triangle) in the figure has two angles marked;
we will argue that these angles must both be at least $\frac{\pi}{2}$ in any piecewise-Euclidean
or piecewise-hyperbolic metric, a contradiction.
Each vertex of the $2$-simplex, being a vertex of $L$, corresponds to a marked graph of groups.
That marked graph of groups is suggested in the figure. Each solid star vertex of the graph of groups
(of which there are four in the first figure, three in the second, and none in the third)
\emph{must} correspond to a vertex with nontrivial vertex group.
(A more careful analysis of the complex as in~\cite{Bridson} allows one to argue that the
spine of Outer Space for $F_n$ with $n \ge 3$ is not CAT$(0)$; but we make no attempt to reproduce
Bridson's result here.)
Similarly, each unstarred vertex \emph{must} have trivial vertex group.
The ellipsis in each figure indicates a point where one \emph{may} attach a reduced marked graph of groups
in order to obtain a marked graph of groups representing a vertex of $L(G,\mathscr{A})$
when $(G,\mathscr{A})$ is more complicated than the base case.

\begin{figure}
  \begin{center}
    \def\svgwidth{0.6\columnwidth}
    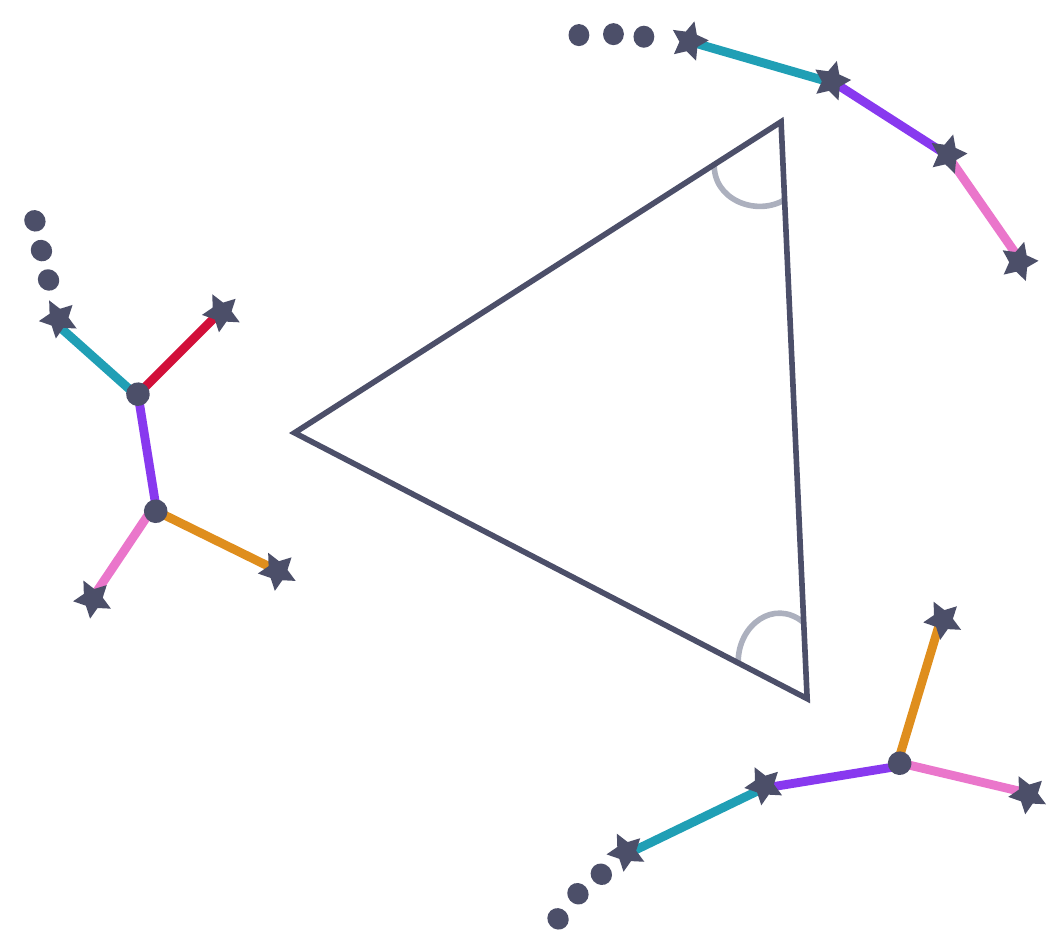
  \end{center}
  \caption{A $2$-simplex in $L$, two of whose angles must but cannot be at least $\frac{\pi}{2}$.}\label{badtri1}
\end{figure}

\begin{figure}[t]
  \begin{center}
    \def\svgwidth{0.6\columnwidth}
    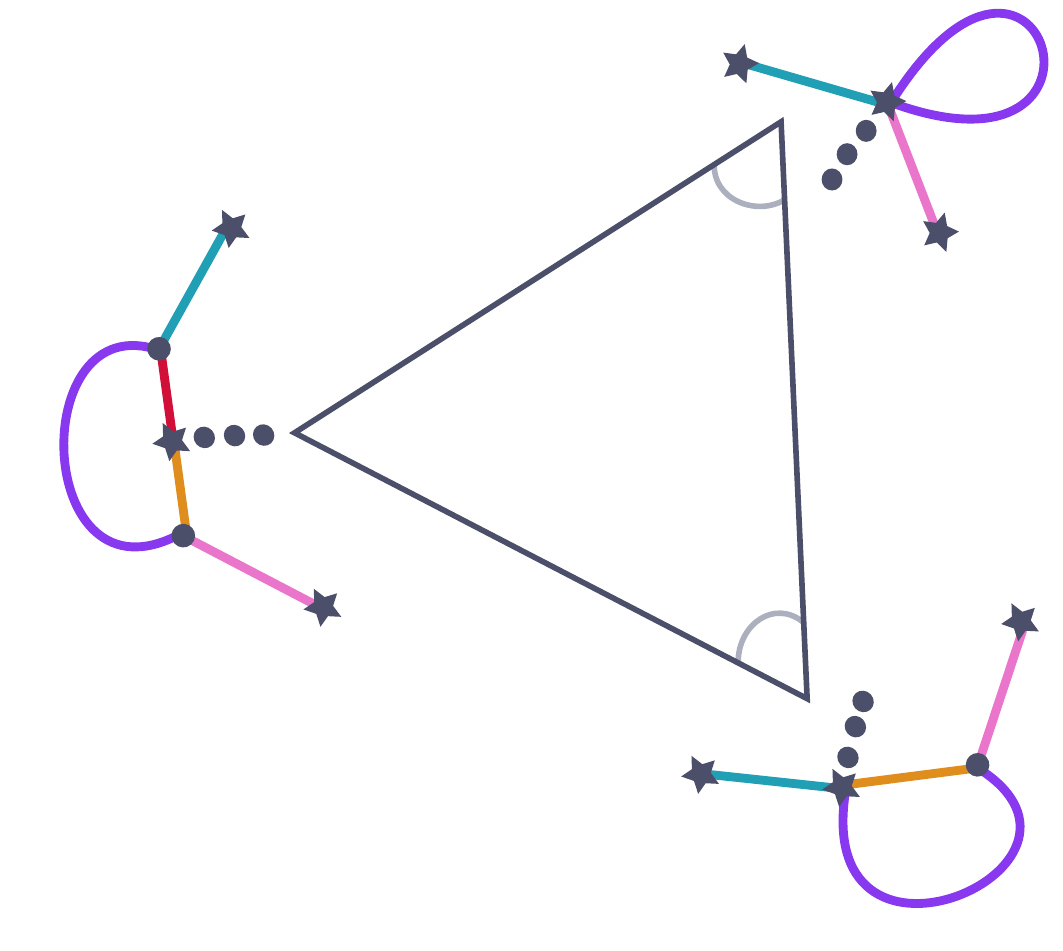
  \end{center}
  \caption{A $2$-simplex in $L$, two of whose angles must but cannot be at least $\frac{\pi}{2}$.}\label{badtri2}
\end{figure}

\begin{figure}
  \begin{center}
    \def\svgwidth{0.6\columnwidth}
    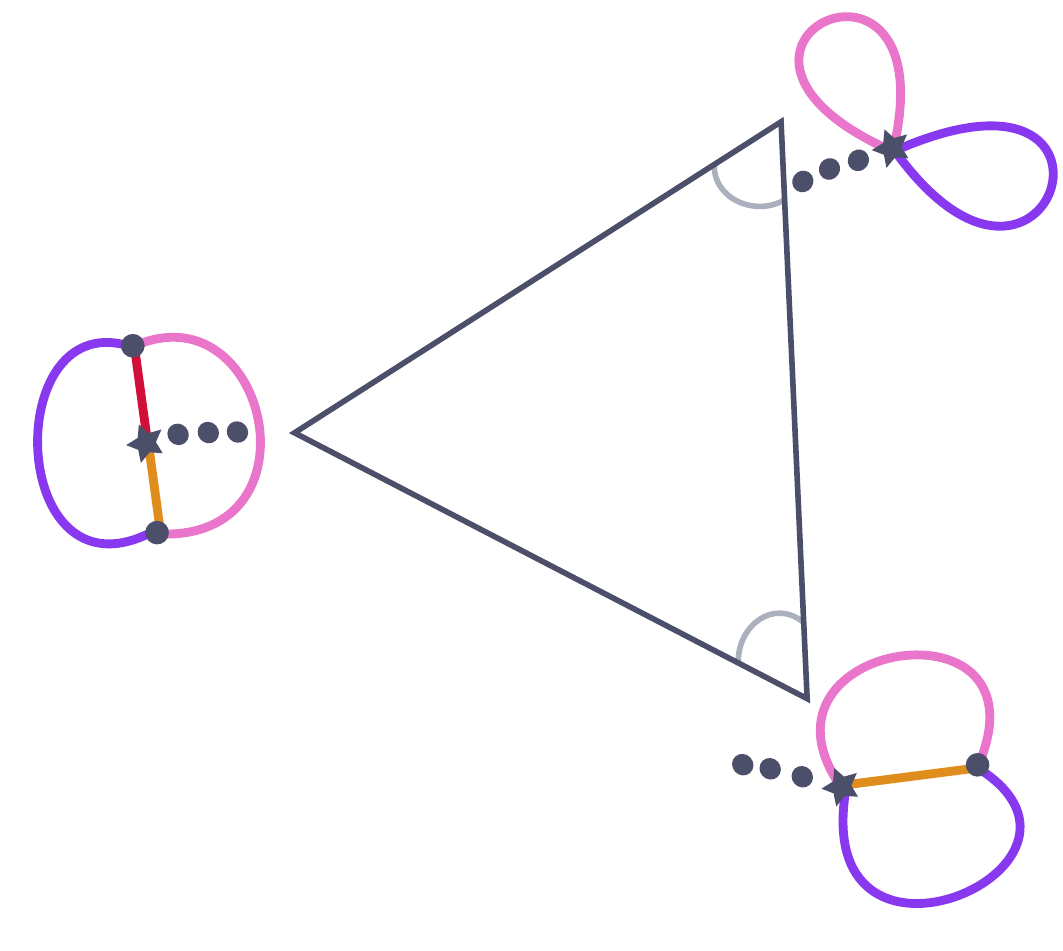
  \end{center}
  \caption{A $2$-simplex in $L$, two of whose angles must but cannot be at least $\frac{\pi}{2}$.}\label{badtri3}
\end{figure}

Thus in each figure, the topmost vertex of the $2$-simplex corresponds to a reduced marked graph of groups.
The other two marked graphs of groups are obtained from it by blowing up an ideal forest
comprising two ideal edges, each of which has ``size two'', in the sense that,
writing the ideal edge as $(\alpha, \mathcal{G}_\alpha)$,
the set $\alpha$ may be taken
(up to complementation in the case that the unfilled starred vertex has trivial vertex group)
to contain two directions.
The two ideal edges are \emph{disjoint:} one, the ideal edge which becomes the gold edge
in the blown up marked graph of groups to the left in each figure,
contains a direction in the orbit of (one orientation of) the pink edge
and one orientation of the purple edge,
while the other ideal edge, which becomes the red edge in figure,
contains a direction in the orbit of the teal edge and the other orientation of the purple edge
in the first two figures, and the other orientation of both the purple and pink edges in the third.

We now turn to showing that in each figure, the indicated angles must
be at least $\frac{\pi}{2}$.
First consider the angle at the minimal marked graph of groups, at the top of each figure.
That angle, in the link of the minimal marked graph of groups,
connects the red ideal edge to the gold ideal edge.
However, because of the presence of the starred vertex where we are blowing up,
these ideal edges are actually \emph{underspecified.}
Because the vertex group is nontrivial, there are at least \emph{two}
choices of ``gold'' or ``red'' ideal edge; in fact there is one for each element of the
corresponding vertex group.
Blowing up two distinct ``gold'' or ``red'' ideal edges
yield distinct marked graphs of groups which differ by an element of $\out(G,\mathscr{A})$.
What's more, there is in fact an element of $\out(G,\mathscr{A})$
taking each \emph{pair} of chosen gold and red edges to any other pair.
Each ``red'' ideal edge is compatible with \emph{each} ``gold'' ideal edge,
so in the link there is a $4$-cycle containing the edge
from our chosen red ideal edge to our chosen gold ideal edge.
Because $\out(G,\mathscr{A})$ acts transitively on the set of pairs of red and gold ideal edges,
in the link, each of the angles must be equal.
Because they form a $4$-cycle, we see that all of these angles must be at least $\frac{\pi}{2}$
in any piecewise-Euclidean or piecewise-hyperbolic metric on $L(G,\mathscr{A})$.

Now consider the other marked angle, at the bottom of each figure.
Because the marked graph of groups is intermediary in the $2$-simplex,
one of the other marked graphs of groups is obtained by blowing up an ideal forest
(the ``red'' ideal edge), while the other is obtained by collapsing a collapsible forest
(the ``gold'' edge). In fact, there are \emph{three} collapsible edges in each figure;
any edge incident to the unstarred vertex.
In fact, the ``downward'' link of this vertex is comprised of these three edges,
each of which is by itself a maximal collapsible forest.
In any of these collapses, it is still possible to blow up the ``red'' ideal edge,
in fact \emph{any} choice of ``red'' ideal edge,
so the link of this marked graph of groups contains an embedded bipartite graph:
one half comprises the red ideal edges
and the other the collapsible edges.
Like all bipartite graphs, this graph contains many $4$-cycles.
Again by $\out(G,\mathscr{A})$-action, the angle length assigned to each of these edges
depends only (at most) upon the collapsible edge and not the ideal one.
Since in each figure at least two of the collapsible edges return us to a marked graph of groups
which is combinatorially identical to the minimal marked graph of groups in the figure
(with the possible exception of the vertex groups),
since the sum $2\theta_1 + 2\theta_2$ of angles around this $4$-cycle,
must be at least $2\pi$ in any piecewise-Euclidean or piecewise-hyperbolic metric on $L(G,\mathscr{A})$,
at least one of $\theta_1$ or $\theta_2$ must be at least $\frac{\pi}{2}$.
We may assume that our pictured $2$-simplex contains that angle.

This completes the proof; in any piecewise-Euclidean or piecewise-hyperbolic CAT$(0)$ metric on
$L(G,\mathscr{A})$, the pictured $2$-simplex must both have angle sum at most $2\pi$
but because two of its angles are at least $\pi$ and the third must be nonzero,
we have a contradiction.
Therefore no $\out(G,\mathscr{A})$-equivariant such metric exists.

%% file: figures/badtriangles1.pdf_tex
\begingroup%
  \makeatletter%
  \providecommand\color[2][]{%
    \errmessage{(Inkscape) Color is used for the text in Inkscape, but the package 'color.sty' is not loaded}%
    \renewcommand\color[2][]{}%
  }%
  \providecommand\transparent[1]{%
    \errmessage{(Inkscape) Transparency is used (non-zero) for the text in Inkscape, but the package 'transparent.sty' is not loaded}%
    \renewcommand\transparent[1]{}%
  }%
  \providecommand\rotatebox[2]{#2}%
  \newcommand*\fsize{\dimexpr\f@size pt\relax}%
  \newcommand*\lineheight[1]{\fontsize{\fsize}{#1\fsize}\selectfont}%
  \ifx\svgwidth\undefined%
    \setlength{\unitlength}{510.23622047bp}%
    \ifx\svgscale\undefined%
      \relax%
    \else%
      \setlength{\unitlength}{\unitlength * \real{\svgscale}}%
    \fi%
  \else%
    \setlength{\unitlength}{\svgwidth}%
  \fi%
  \global\let\svgwidth\undefined%
  \global\let\svgscale\undefined%
  \makeatother%
  \begin{picture}(1,0.89444444)%
    \lineheight{1}%
    \setlength\tabcolsep{0pt}%
    \put(0,0){\includegraphics[width=\unitlength,page=1]{badtriangles1.pdf}}%
  \end{picture}%
\endgroup%

%% file: figures/badtriangles2.pdf_tex
\begingroup%
  \makeatletter%
  \providecommand\color[2][]{%
    \errmessage{(Inkscape) Color is used for the text in Inkscape, but the package 'color.sty' is not loaded}%
    \renewcommand\color[2][]{}%
  }%
  \providecommand\transparent[1]{%
    \errmessage{(Inkscape) Transparency is used (non-zero) for the text in Inkscape, but the package 'transparent.sty' is not loaded}%
    \renewcommand\transparent[1]{}%
  }%
  \providecommand\rotatebox[2]{#2}%
  \newcommand*\fsize{\dimexpr\f@size pt\relax}%
  \newcommand*\lineheight[1]{\fontsize{\fsize}{#1\fsize}\selectfont}%
  \ifx\svgwidth\undefined%
    \setlength{\unitlength}{510.23622047bp}%
    \ifx\svgscale\undefined%
      \relax%
    \else%
      \setlength{\unitlength}{\unitlength * \real{\svgscale}}%
    \fi%
  \else%
    \setlength{\unitlength}{\svgwidth}%
  \fi%
  \global\let\svgwidth\undefined%
  \global\let\svgscale\undefined%
  \makeatother%
  \begin{picture}(1,0.89444444)%
    \lineheight{1}%
    \setlength\tabcolsep{0pt}%
    \put(0,0){\includegraphics[width=\unitlength,page=1]{badtriangles2.pdf}}%
  \end{picture}%
\endgroup%

%% file: figures/badtriangles3.pdf_tex
\begingroup%
  \makeatletter%
  \providecommand\color[2][]{%
    \errmessage{(Inkscape) Color is used for the text in Inkscape, but the package 'color.sty' is not loaded}%
    \renewcommand\color[2][]{}%
  }%
  \providecommand\transparent[1]{%
    \errmessage{(Inkscape) Transparency is used (non-zero) for the text in Inkscape, but the package 'transparent.sty' is not loaded}%
    \renewcommand\transparent[1]{}%
  }%
  \providecommand\rotatebox[2]{#2}%
  \newcommand*\fsize{\dimexpr\f@size pt\relax}%
  \newcommand*\lineheight[1]{\fontsize{\fsize}{#1\fsize}\selectfont}%
  \ifx\svgwidth\undefined%
    \setlength{\unitlength}{510.23622047bp}%
    \ifx\svgscale\undefined%
      \relax%
    \else%
      \setlength{\unitlength}{\unitlength * \real{\svgscale}}%
    \fi%
  \else%
    \setlength{\unitlength}{\svgwidth}%
  \fi%
  \global\let\svgwidth\undefined%
  \global\let\svgscale\undefined%
  \makeatother%
  \begin{picture}(1,0.89444444)%
    \lineheight{1}%
    \setlength\tabcolsep{0pt}%
    \put(0,0){\includegraphics[width=\unitlength,page=1]{badtriangles3.pdf}}%
  \end{picture}%
\endgroup%